\documentclass[11pt]{article}
\usepackage[margin =1in]{geometry}
\usepackage{amssymb,amsthm}
\usepackage{amsmath}
\usepackage{enumerate}
\usepackage{hyperref}
\usepackage{cite}
\usepackage{color}
\usepackage{bm}

\usepackage{mathtools}

\numberwithin{equation}{section}

\newtheorem{thm}{Theorem}[section]

\newtheorem{lem}[thm]{Lemma}
\newtheorem{cor}[thm]{Corollary}

\newtheorem{remark}{Remark}

\theoremstyle{remark}

\renewcommand{\a}{\alpha} 
\renewcommand{\L}{L} 

\renewcommand{\b}{\bm{b}}
\renewcommand{\d}{\bm{d}}
\newcommand{\e}{\bm{e}}

\newcommand{\h}{\bm{h}}
\newcommand{\p}{\bm{p}}
\newcommand{\q}{\bm{q}}
\renewcommand{\r}{\bm{r}}
\renewcommand{\v}{\bm{v}}
\newcommand{\w}{\bm{w}}
\newcommand{\x}{\bm{x}}
\newcommand{\y}{\bm{y}}
\newcommand{\z}{\bm{z}}
\newcommand{\bz}{\bm{0}}

\newcommand{\prox}{\mbox{prox}} 
\newcommand{\argmin}{\operatornamewithlimits{argmin}} 

\renewcommand\epsilon\varepsilon
\newcommand{\norm}[1] {\left \| #1 \right \|} 
\newcommand{\eg}{\emph{e.g.}}
\newcommand{\dist}{{\rm dist}} 
\newcommand{\R}{\bm{\mathrm{R}}} 
\newcommand{\oR}{\overline \R} 

\newcommand\Span{{\rm span}}

\usepackage{mathtools}

\usepackage[boxruled]{algorithm2e}

	\title{Potential-based analyses of first-order methods for constrained and composite optimization}
	
	\author{Courtney Paquette\thanks{Department of Combinatorics and Optimization, University of Waterloo, Waterloo, ON, N2L 3G1, Canada;
\texttt{cypaquette.github.io/}. Research of Paquette was supported by
NSF DMS award 1803289 (Postdoctoral Fellowship). }
		\and 
		Stephen Vavasis\thanks{Department of Combinatorics and Optimization, University of Waterloo, Waterloo, ON, N2L 3G1, Canada;
		\url{https://www.math.uwaterloo.ca/\~vavasis}. Research of Vavasis was supported by a Discovery Grant from the Natural Science and Engineering Research Council (NSERC) of Canada.}}
	
	\date{\today}
	
\begin{document}
	\maketitle
	
	\begin{abstract}
    We propose potential-based analyses for first-order algorithms applied to constrained and composite minimization problems.  We first propose ``idealized'' frameworks for algorithms in the strongly and non-strongly convex cases and argue based on a potential that methods following the framework achieve the best possible rate.  Then we show that the geometric descent (GD) algorithm by Bubeck et al.\ as extended to the constrained and composite setting by Chen et al.\ achieves this rate using the potential-based analysis for the strongly convex case.  Next, we extend the GD algorithm to the case of non-strongly convex problems.  We show using a related potential-based argument that our extension achieves the best possible rate in this case as well.  The new GD algorithm achieves the best possible rate in the nonconvex case also.  We also analyze accelerated gradient using the new potentials.

We then turn to the special case of a quadratic function with a single ball constraint, the famous trust-region subproblem.  For this case, the first-order trust-region Lanczos method by Gould et al.\ finds the optimal point in an increasing sequence of Krylov spaces.  Our results for the general case immediately imply convergence rates for their method in both the strongly convex and non-strongly convex cases.  We also establish the same convergence rates for their method using arguments based on Chebyshev polynomial approximation.  To the best of our knowledge, no convergence rate has previously been established for the trust-region Lanczos method.
\end{abstract}
	
\section{Composite problems}
We consider unconstrained problems of the form
\begin{equation}\label{eq:composite_problem} \min_{\x} F(\x) := f(\x) + \Psi(\x) , \end{equation}
where $f : \R^n \to \R$ is an $\L$-smooth
function and $\Psi: \R^n \to \oR$ is a proper, closed, convex function.  Objectives of this
form are often called ``composite'' functions. In much of the paper, $f$ will also be convex or $\a$-strongly convex.  Recall
that an {\em $\L$-smooth} function $f$ is differentiable
and satisfies the following inequality for all
$\x,\y\in\R^n$:
\[ |f(\x)-f(\y)-\nabla f(\y)^T(\x-\y)|\le (\L/2) \Vert \x-\y\Vert^2. \]
For $\a>0$, an {\em $\a$-strongly convex} function $f$ satisfies,
\[ f(\y)\ge f(\x)+\bm{g}^T(\y-\x)+(\a/2)\Vert \x-\y\Vert^2, \]
for all $\x,\y\in\R^n$, where $\bm{g}$ is any subgradient of $f$ at $\x$.

The function $\Psi$ is assumed to be a ``simple'' convex function in the sense
that there is an efficient algorithm to evaluate the prox operator (see
below for the definition of prox) for $\Psi$.  Common examples include the
$1$-, $2$- and $\infty$-norms and the matrix nuclear norm.  Other examples
include indicator functions for simple closed convex sets.  In the case that
$\Psi$ is an indicator function $\mathrm{i}_\Omega$ of a closed convex nonempty set $\Omega$, $\min_{\x} F(\x)$ is
equivalent to the constrained problem $\min_{\x\in\Omega} f(\x)$.

Many commonly occurring problems fit into the framework of composite
optimization including compressive sensing, robust PCA, basis pursuit
denoising, and the trust region problem.  For this reason, it has attracted
substantial attention in the literature.  For large problems in which interior-point
methods would be intractable, the problems are usually solved with
first-order methods.  On each iteration, a typical first-order method
requires evaluation of the gradient of $f$ and a prox computation involving
$\Psi$.  The first method in this class is due to Nesterov \cite{intro_lect} and
is optimal in the sense that the method achieves efficiency guarantees matching the best possible complexity estimates for the (strongly) convex composite setting.  In fact,  it is optimal even for the special case that $\Psi\equiv 0$.
Nesterov's lower bound assumes that the objective function
$f$ is accessible only via a function/gradient oracle.  In the strongly convex case,
the best possible reduction in the residual is a factor of $(1-{\rm const}\cdot\sqrt{\a/\L})^k$ after $k$ iterations, and the best possible
reduction is a factor ${\rm const}/k^2$ in the non-strongly convex case after $k$ 
iterations.  Nesterov's algorithm matches these bounds.

Despite the optimality of Nesterov's method, it has some limitations that
have been addressed in follow-up methods.  Nesterov himself introduced
at least two other optimal methods, and more recently Bubeck, Lee and Singh
\cite{GD_Bubeck} proposed the Geometric Descent (GD) algorithm for unconstrained
strongly-convex minimization.  The GD algorithm
is often faster in practice than Nesterov's original method and also has an elegant geometric interpretation.
The GD algorithm has been extended to composite functions by 
Chen, Ma and Liu \cite{GD_chen}.  

In this paper, we analyze these algorithms using a potential function.  We start by introducing an idealized algorithm (IA) framework in Section~\ref{sec:IA1}, not implementable in general, that achieves the optimal convergence rate for composite functions in which $f$ is strongly convex.  Then we explain how Chen's geometric descent algorithm can be understood as an approximation to the framework in Section~\ref{sec:chen} and thereby also achieves the optimal rate.  We can also analyze Nesterov's first algorithm in this framework in Section~\ref{nest:analysis}.  The potential is a direct extension of previous
work by Karimi and Vavasis \cite{KV}, who considered the (noncomposite) setting when $\Psi\equiv 0$ and $f$ is strongly convex. 

We then consider the non-strongly convex case in Section~\ref{sec:IA-cvx}.  In this case, a variant of the idealized algorithm also achieves the optimal rate as we show using a different potential function.  We then propose a new Geometric Descent algorithm for this case and fit it into the framework.

It turns out  that the new Geometric Descent algorithm also solves nonconvex problems, and achieves the best possible bound in this case also, although the rate analysis does not involve a potential.  This case is analyzed in Section~\ref{sec:nonconvex}.

For the last part of the paper, we turn to the trust-region subproblem, that is, $\min_{\x\in B} f(\x)$, where $f$ is a quadratic function and $B=\{\x:\Vert\x\Vert\le\Delta\}$, the Euclidean ball of radius $\Delta>0$.  For this problem, an optimal
algorithm has been proposed by Gould, Lucidi, Roma and Toint \cite{CG_Lanczos}.  Their method is optimal in a strong sense that it finds the best point in a growing sequence of Krylov subspaces on each iteration.  Our analysis from Sections~\ref{sec:IA1} and \ref{sec:IA-cvx} immediately imply a convergence rate for the trust-region Lanczos method, as explained in Section~\ref{sec:CG-Lanczos}.  A convergence rate can also be established for this method using the classical technique of Chebyshev polynomial approximation, which is the subject of Section~\ref{sec:cheb}. 

\section{Prox operation}

The notation we follow is standard. Throughout we consider a Euclidean space, denoted by $\R^n$, with an inner product and an induced norm $\norm{\cdot}$. Given a closed nonempty convex set $\Omega$ in $\R^n$, the \textit{distance} and \textit{projection} of a point $\x$ onto $\Omega$ are given by 
\[ \text{dist}(\x,\Omega) = \inf_{\y \in \Omega} \norm{\x-\y}, \quad \text{proj}_\Omega(\x) = \argmin_{\y \in \Omega} \norm{\y-\x},\]
respectively. The extended real-line is the set $\oR = \R \cup \{\pm \infty\}$. The \textit{domain} and \textit{epigraph} of any function $g: \R^n \to \oR$ are the sets
\[\text{dom}~g := \{\x \in \R^n \, : \, g(\x) < \infty\}, \quad \text{epi}~g := \{(\x,r) \in \R^n \times \R \, :  g(\x) \le r\},\]
respectively. We say a function $g$ is \textit{closed} if its epigraph, epi~$g$ is a closed set. Throughout this paper, we will assume all functions are \textit{proper}, namely, they have nonempty domains and never take the value $-\infty$. The indicator of a set $\Omega \subseteq \R^n$ denoted by $\mathrm{i}_\Omega$ is defined to be $0$ on the set $\Omega$ and $\infty$ outside of it.  

Let $\Psi:\R^n\rightarrow \oR$ be a closed, proper convex function. For any $t > 0$, the prox operator (see, e.g. Rockafellar and Wets
\cite{RW98}) is defined by:
\[\prox_{t\Psi}(\x)=\argmin_{\z}\{\Psi(\z)+\Vert \x-\z\Vert^2/(2t)\}. \]
This operator is well defined, i.e., the minimizer exists and is unique for all $\x\in \R^n$. In the special case, $\Psi(\x) = \mathrm{i}_\Omega(\x)$, an indicator of a nonempty, closed convex set $\Omega$, the proximal mapping of $\Psi(\x)$ for any $t>0$ is the projection $\x \mapsto \text{proj}_\Omega(\x)$.

For any $t > 0$, the \textit{forward-backward step} for a composite problem $F(\x)=f(\x)+\Psi(\x)$ is defined to be $\prox_{t\Psi}(\x - t\nabla f(\x))$.  Note the identity
\begin{equation}
    \label{eq:forwbackstepid}\prox_{t\Psi}(\x - t\nabla f(\x)) = \argmin_{\z} \left \{ f(\x) +
  \nabla f(\x)^T(\z-\x) + \frac{1}{2t} \norm{\z-\x}^2 + \Psi(\z) \right \}.
\end{equation}
The \textit{prox-gradient} mapping
is defined by
\begin{equation} \label{eq:prox_grad_defn}
G_t(\x) := t^{-1} (\x - \prox_{t \Psi}(\x-t\nabla f(\x))).
 \end{equation}

\subsection{First-order stationary points for composite problems} 
Let us explain the goal of algorithms for solving \eqref{eq:composite_problem} and its relationship to the proximal mapping. Given a convex function $g: \R^n \to \oR$, a vector $\v$ is called a \textit{subgradient} of $g$ at a point $\x \in \text{dom}~g$ if the inequality
\begin{equation} \label{eq:convex_subdifferential} g(\y) \ge g(\x) + \v^T(\y-\x) \quad \text{holds for all $\y \in \R^n$.} \end{equation}
The set of all subgradients of $g$ at $\x$ is denoted by $\partial g(\x)$, and is called the \textit{subdifferential} of $g$ at $\x$. In (strongly) convex optimization, standard complexity bounds are derived to guarantee either small function values or near-stationary points, $\text{dist}(\bz, \partial g(\x)) < \varepsilon$. When $\bz\in\partial g(\x)$, $\x$ is a stationary point and first-order optimality conditions are satisfied. For a convex problem, such an $\x$ is globally minimal.  In the absence of convexity, it is natural to seek points $\x$ that are only first-order stationary. One makes this notion precise through subdifferentials (or generalized derivatives), which have an explicit formulation for the composite problem class. We recall the following relevant definitions from Mordukhovich \cite{Mord_1} and Rockafellar and Wets \cite{RW98}.

Consider an arbitrary function $g: \R^n \to \oR$ and a point $\bar{\x}$ with $g(\bar{\x})$ finite. The \textit{Fr\'{e}chet subdifferential} of $g$ at $\bar{\x}$, denoted $\hat{\partial} g({\bm{\bar{x}}})$, is the set of all vectors $\bm{v}$ satisfying
\[ g(\x) \ge g({\bm{\bar{x}}}) + {\bm{v}}^T(\x-\bar{\x}) + o(\norm{\x-\bar{\x}}) \quad \text{as $\x \to \bar{\x}$.} \]
Hence $\v \in \hat{\partial}g(\bar{\bm{x}})$ holds precisely when the affine function $\x \mapsto g(\bar{\x}) + \v^T(\x-\bar{\x})$ underestimates $g$ up to first-order near $\bar{\x}$. In general, the limit of Fr\'{e}chet subgradients ${\bm{v}_i} \in \hat{\partial} g(\x_i)$, along a sequence $\x_i \to \bar{\x}$ may not be a Fr\'{e}chet subgradient at the limiting point $\bar{\x}$. One can formally enlarge the Fr\'{e}chet subdifferential and define the \textit{limiting subdifferential} of $g$ at $\bar{\x}$, denoted $\partial g(\bar{\x})$, to consist of all vectors $\bm{v}$ for which there exists sequences $\x_i$ and $\bm{v}_i$ satisfying $\bm{v}_i \in \partial g(\x_i)$ and $(\x_i, g(\x_i), \bm{v}_i) \to (\bar{\x}, g(\bar{\x}), \bm{v})$. 

For convex functions $g$, the subdifferentials $\hat{\partial}g(\x)$ and $\partial g(\x)$ coincide with the subdifferential in the sense of convex analysis \eqref{eq:convex_subdifferential}, while for $C^1$-smooth functions $g$, they consist only of the gradient $\nabla g(\x)$. Similarly for the composite setting when the objective is nonconvex \eqref{eq:composite_problem} the two subdifferentials coincide and admit an intuitive summation rule \cite[Corollary 10.9]{RW98}
\[\partial (f + \Psi)(\x) = \nabla f(\x) + \partial \Psi(\x).\]

The prox-gradient mapping \eqref{eq:prox_grad_defn} measures near-optimality of composite problems \eqref{eq:composite_problem}. One makes this idea precise by considering the first-order optimality conditions for the mapping $\z \to  f(\x) + \nabla f(\x)^T(\z-\x) + \tfrac{1}{2t} \norm{\z-\x}^2 + \Psi(\z)$, namely, 
\begin{equation} \label{eq:prox_grad_optimality} G_t(\x) - \nabla f(\x) \in \partial \Psi( \prox_{t\Psi}(\x-t\nabla f(\x)) ); \end{equation}
hence $G_t(\x)$ is neither a gradient nor a subgradient of either $F(\x)$ or $F(\prox_{t\Psi}(\x-t\nabla f(\x)))$. Despite this, there is a natural relationship between the prox-gradient and
near-stationarity of $F$. 
\begin{lem} \label{lem:near_stationarity} Suppose the function $F: \R^n \to \oR$ is defined by $F = f + \Psi$ where $f: \R^n \to \R$ is an $\L$-smooth function and $\Psi: \R^n \to \oR$ is a proper, closed convex function. Then the prox-gradient mapping is a measure of near-stationarity, namely, 
\begin{equation}\label{eq:near_stationarity} \dist(0, \partial F(\text{\rm prox}_{\Psi/L}(\x-\nabla f(\x)/L))) \le 2 \norm{G_{1/L}(\x)}. \end{equation}
Furthermore, for all $t>0$, $G_t(\x) = \bz$ if and only if 
$\bz\in \partial F(\x)$.
\end{lem}
Note that in the case that $f$ is convex, the final condition of the lemma, $\bz\in\partial F(\x)$, holds if and only if $\x$ minimizes $F$.
\begin{proof}
  The optimality conditions for the mapping $\z \to f(\x) + \nabla f(\x)^T(\z-\x) + \frac{L}{2}\norm{\z-\x}^2 + \Psi(\z)$ with $t = 1/L$ \eqref{eq:prox_grad_optimality} yield the following:
  \begin{align*}
  G_{1/L}(\x) - \nabla f(\x) + \nabla f(\prox_{\Psi/L}(\x-\nabla f(\x)/L)) \in \partial F(\prox_{\Psi/L}(\x-\nabla f(\x)/L))
  \end{align*}
  The triangle inequality applied to the left-hand side together with $L$-smoothness of $f$ gives the desired inequality \eqref{eq:near_stationarity}. 
  
  For the second claim, observe from the definition of $G_t$ that $G_t(\x)=\bz$ if and only if
  $\prox_{t\Psi}(\x-t\nabla f(\x))=\x$.  Since the objective function on the right-hand side of \eqref{eq:forwbackstepid} is convex in
  $\z$, it has a minimizer at $\x$ if and only if $\bz$ lies in the subdifferential of the right-hand side with respect to $\z$ evaluated at $\x$, which is seen to be $\nabla f(\x)+\partial\Psi(\x)$.  This proves $G_t(\x)=\bz$ if and only
  if $\bz\in\partial F(\x)$.
\end{proof}
In summary, the algorithms we consider for the nonconvex composite setting aim to find stationary points of $F$, i.e. those points $\x$ satisfying $\bz \in \partial F(\x)$ whereas the algorithms applied to (strongly) convex composite objectives seek points $\x$ where the objective function is small. It is also worth noting that a point $\x$ is stationary for $F$ if and only if the directional derivatives of $F$ are nonnegative in every direction \cite[Proposition 8.32]{RW98}. 

\subsection{Key inequalities involving the prox-gradient mapping}
For an $\L$-smooth function $f$, define the following point:
\[\bar{\x} := \x - \frac{G_{1/\L}(\x)}{\L},\]
and for an $\alpha$-strongly convex $\L$-smooth $f$, 
\[\bar{\bar \x} := \x - \frac{G_{1/\L}(\x)}{\a}.\]
It follows from the definitions that $\bar{\x} = \prox_{\Psi/\L}(\x-\nabla f(\x)/L)$, i.e., $\bar{\x}$ represents the forward-backward step starting from $\x$. We next state some key inequalities regarding the
prox-gradient. It is known for any $L$-smooth function, $f$, the following inequality holds
\[ f(\x -\nabla f(\x)/\L) \le f(\x) - \frac{1}{2\L} \norm{\nabla
    f(\x)}^2. \]
For a nonsmooth, composite function, an analogous bound with the prox-gradient holds.
\begin{lem} \label{lem:smooth} Suppose $F := f + \Psi$ where $f : \R^n \to \R$ is a convex, $L$-smooth
  function and $\Psi : \R^n \to \oR$ is closed, proper and convex. Then the following inequality holds
\begin{align*}
F(\bar{\x}) \le F(\y) + G_{1/\L}(\x)^T(\x-\y) - \frac{1}{2\L}
  \norm{G_{1/\L}(\x)}^2, \qquad \text{for all $\x,\y \in \R^n$}.
\end{align*}
In particular if $\y = \x$, the result simplifies to
\[ F(\bar{\x}) \le F(\x) - \frac{1}{2\L} \norm{G_{1/\L}(\x)}^2. \]
The latter inequality holds even if $f$ is not convex.
\end{lem} 

\begin{proof} The proof is standard (see \eg \cite[Corollary 2.3.2]{intro_lect}). By
  $L$-smoothness of the function $f$, we know
\begin{equation} \label{eq:L_smooth_f} f(\bar{\x}) \le f(\x)  - \frac{1}{\L} \nabla
  f(\x)^T G_{1/\L}(\x) + \frac{1}{2\L} \norm{G_{1/\L}(\x)}^2. 
\end{equation}
Using the convexity of $f$ with the above inequality, we derive
\begin{equation} \begin{aligned} \label{eq:L_smooth_f_1}
F(\bar{\x})&\le \Psi(\bar{\x}) + f(\x) -\frac{1}{L}\nabla f(\x)^TG_{1/L}(\x) + \frac{1}{2L}\Vert G_{1/L}(\x)\Vert^2 \\
&\le \Psi(\bar{\x}) + f(\y) + \nabla f(\x)^T(\x-\y) -
  \frac{1}{\L} \nabla f(\x)^T G_{1/\L}(\x) + \frac{1}{2\L}
  \norm{G_{1/\L}(\x)}^2\\
&\le F(\y) + [G_{1/\L}(\x) - \nabla f(\x)]^T (\bar{\x}-\y) + \nabla
  f(\x)^T(\x-\y)\\
&\qquad \mbox{}- \frac{1}{\L} \nabla f(\x)^T G_{1/\L}(\x) + \frac{1}{2\L}
  \norm{G_{1/\L}(x)}^2\\
&= F(\y) + [G_{1/\L}(\x)-\nabla f(\x)]^T \left(\x-\y-\tfrac{1}{\L} G_{1/\L}(\x)\right) +
  \nabla f(\x)^T (\x-\y)\\
&\qquad \mbox{}- \frac{1}{\L} \nabla f(\x)^T G_{1/\L}(\x) +
  \frac{1}{2\L} \norm{G_{1/\L}(\x)}^2,
\end{aligned}
\end{equation}
where the third inequality follows by convexity of $\Psi$ and
$G_{1/\L}(\x) - \nabla f(\x) \in \partial \Psi(\bar{\x})$. Simplifying
the last inequality yields the desired result. 

Convexity of $f$ was used only in the second line of the above chain of inequalities.  In the case that $\x=\y$, the third line is trivially true, thus proving the last claim of the lemma.
\end{proof}

A lower quadratic bound holds for smooth, strongly convex functions; an analogous bound also holds in the composite setting. 

\begin{lem} \label{lem:strongly_convex} Suppose $F := f + \Psi$ where $f : \R^n \to \R$ is
  $\a$-strongly convex and $L$-smooth function and $\Psi: \R^n \to \R$ is a proper, closed, convex function. Then the
  following inequality holds
\begin{align} \label{eq:strong_cvx}
F(\y) \ge F(\bar{\x}) + G_{1/\L}(\x)^T(\y-\x) + \frac{1}{2\L}
  \norm{G_{1/\L}(\x)}^2 + \frac{\a}{2} \norm{\y-\x}^2.
\end{align}
\end{lem}

\begin{proof} The proof is standard (\eg, see \cite[Corollary 2.3.2]{intro_lect}). From
  \eqref{eq:L_smooth_f} with strong convexity of $f$, we have 
\begin{align*}
f(\bar{\x}) \le f(\y) + \nabla f(\x)^T(\x-\y) - \frac{\a}{2} \norm{\x-\y}^2 -
  \frac{1}{\L} \nabla f(\x)^T G_{1/\L}(\x) + \frac{1}{2\L} \norm{G_{1/\L}(\x)}^2.
\end{align*}
Following the argument in \eqref{eq:L_smooth_f_1}, we conclude 
\begin{align*}
F(\bar{\x}) &\le F(\y) + [G_{1/\L}(\x)-\nabla f(\x)]^T \left(\x-\y-\tfrac{1}{\L} G_{1/\L}(\x)\right) +
  \nabla f(\x)^T (\x-\y)\\
&\qquad \mbox{} - \frac{1}{\L} \nabla f(\x)^T G_{1/\L}(\x) +
  \frac{1}{2\L} \norm{G_{1/\L}(\x)}^2 - \frac{\a}{2} \norm{\x-\y}^2.
\end{align*}
The result follows from simplifying the above expression. 
\end{proof}

\section{Geometric Lemmas}

The starting point for our discussion is a description of the minimum enclosing ball of the intersection of two balls. Indeed, the foundations for the convergence analysis of both the Bubeck et al.\ geometric descent \cite{GD_Bubeck} and Drusvyatskiy et al.\ optimal averaging \cite{GD_Dima} algorithms start from this observation.  
\begin{lem}[Intersection of two
  balls from \cite{GD_Bubeck, GD_Dima}] \label{lem:intersection_ball_ideal} Suppose $\z, \y \in \R^n$. Let
  $\delta, \rho, \sigma$ be three non-negative scalars such that
  $\delta \le \norm{\z-\y}$. Suppose $\lambda \in [0,1]$ and
\[\z' = (1-\lambda) \z + \lambda \y.\]
Then
\[ B(\z,\rho) \cap B(\y, \sigma) \subset B(\z',\xi), \]
where $\xi^2 = (1-\lambda) \rho^2 + \lambda \sigma^2 - \lambda
(1-\lambda) \delta^2$. The radius $\xi$ is guaranteed to be positive
whenever $\rho + \sigma \ge \norm{\z-\y}$. 
\end{lem}

In essence, to find the minimum ball enclosing the intersection, one optimizes the constant $\lambda$ in the formula for the radius $\xi$. This leads to the corresponding lemma as observed by \cite{KV}.

\begin{lem}[Intersection of two balls with chosen $\lambda$] \label{lem:intersection_balls} Let $\z,
  \y, \rho, \sigma, \delta$ be as in the preceding lemma. Suppose the radii satisfy $\rho + \sigma \ge \delta$ and
  $|\rho^2-\sigma^2| \le \delta^2$. If we set 
\[ \lambda = \frac{\delta^2 + \rho^2 - \sigma^2}{2 \delta^2} \quad
  \text{and} \quad \z^* = (1-\lambda)\z + \lambda \y, \]
then we have
\[ B(\z, \rho) \cap B(\y, \sigma) \subset B(\z^*, \xi), \]
where $\xi^2 = \frac{\rho^2}{2} + \frac{\sigma^2}{2} -
\frac{\delta^2}{4} - \frac{1}{4 \delta^2} \left ( \rho^2-
  \sigma^2 \right )^2$. 
\end{lem}

\begin{proof} Since  $|\rho^2-\sigma^2|/\delta^2 \le 1$, we have $\lambda \in [0,1]$. The
  rest follows from plugging $\lambda$ into the previous lemma.  
\end{proof}

We will be interested when the $\rho$, $\sigma$, and $\delta$ take a particular form as seen in Chen et al.\ \cite{GD_chen} and Bubeck et al.\ \cite{GD_Bubeck}. 

\begin{cor} \label{cor:intersection} Fix points $\z, \y \in \R^n$ and scalars $r_1, r_2 > 0$,
  $\varepsilon \in [0,1]$, and $C \ge 0$. Suppose $\norm{\y-\z} \ge r_2$
  and $\norm{\y-\z} \le \sqrt{r_1^2-\varepsilon r_2^2 - C} +
  \sqrt{(1-\varepsilon) r_2^2 - C}$. If we set 
\begin{equation} \label{eq:value_z}
\z^* = \begin{cases}
(1-\lambda)\z + \lambda \y, & r_1^2 \le 2r_2^2\\
\z, & r_1^2 > 2r_2^2
\end{cases} \quad \quad \text{where} \quad \quad \lambda = \frac{2r_2^2-r_1^2}{2r_2^2},
\end{equation}
then the point $\z^*$ satisfies the following
\[ B \left (\y, \sqrt{r_1^2-\varepsilon r_2^2-C} \right ) \cap B \left (\z,
  \sqrt{(1-\varepsilon) r_2^2-C} \right ) \subset B\left (\z^*,
  \sqrt{(1-\sqrt{\varepsilon})r_1^2 -C} \right ).\]
\end{cor}

\begin{proof} Let $\tilde{\sigma}^2 = r_1^2-\varepsilon r_2^2$ and
  $\tilde{\rho}^2 = (1-\varepsilon)r_2^2$ with corresponding radii
  $\sigma^2 = \tilde{\sigma}^2-C$ and $\rho^2 = \tilde{\rho}^2 -C$,
  respectively. We consider two cases depending on whether
  $\tilde{\sigma}$ is larger than $\tilde{\rho}$. 

Suppose $\tilde{\sigma}^2 \le \tilde{\rho}^2 + r_2^2,$ or equivalently,
$r_1^2 \le 2r_2^2$. In this case, we want
to use the optimal choice of $\lambda$ from Lemma~\ref{lem:intersection_balls}
with the relationship $\delta^2= r_2^2$. To do so, we see $\tilde{\sigma}^2
- \tilde{\rho}^2 \le r_2^2 = \delta^2$ and $\tilde{\rho}^2 - \tilde{\sigma}^2
\le \tilde{\rho}^2 \le r_2^2$. Hence in both cases, we have
\[ |\sigma^2 - \rho^2| = |\tilde{\sigma}^2-\tilde{\rho}^2| \le r_2^2 =
  \delta^2. \]
Moreover, by assumption, $\delta = r_2 \le \norm{\y-\z} \le \sigma +
\rho$. We apply Lemma~\ref{lem:intersection_balls}. Under standard
simplifications, if we set $\z^* = (1-\lambda)\z + \lambda \y$ with
$\lambda = \tfrac{2r_2^2-r_1^2}{2r_2^2}$, then the iterate $\z^*$
satisfies $B(\y, \sigma) \cap B(\z,
\rho) \subset B(\z^*, \xi)$ where 
\begin{equation} \label{eq:value_xi}
\xi^2 = r_1^2 - \varepsilon r_2^2 - \frac{r_1^4}{4r_2^2} - C \le
(1-\sqrt{\varepsilon}) r_1^2 - C.
\end{equation}
The last inequality follows from $a^2 + b^2 \ge 2ab$ applied to $a
= \sqrt{\varepsilon} \cdot r_2$ and $b = \tfrac{r_1^2}{2r_2}$. 

If $\tilde{\sigma}^2 > \tilde{\rho}^2 + r_2^2$, then we can not directly
apply Lemma~\ref{lem:intersection_balls} because we do not necessarily
have a positive $\lambda$. However, we still have sufficient decrease
by setting $\lambda = 0$. The condition $\tilde{\sigma}^2 >
\tilde{\rho}^2 + r_2^2$ occurs if and only if
$r_1^2 > 2 r_2^2$. If we set $z^* = z$, we have $\rho^2 = (1-\varepsilon)r_2^2
- C \le \tfrac{(1-\varepsilon) r_1^2}{2} - C$. The
result follows if $\tfrac{1-\varepsilon}{2} \le
1-\sqrt{\varepsilon}$ or equivalently $0 \le 1/2 - \sqrt{\varepsilon}
+ \tfrac{\varepsilon}{2}$ 
for all $\varepsilon \in [0,1]$ holds. The term $1/2 -
\sqrt{\varepsilon} + \tfrac{\varepsilon}{2}$ is a perfect square,
namely it equals $\tfrac{1}{2}
(\sqrt{\varepsilon}-1)^2$ and thus is always nonnegative. 

\end{proof}

\section{Idealized algorithm for strongly convex functions} \label{sec:IA1}

In this section we present Algorithm~\ref{alg:IA}, the idealized algorithmic framework for minimizing composite functions in which $f$ is $\a$-strongly convex.  This framework is said to be ``idealized'' because it is not implementable in general; several of the steps require prior knowledge of the optimizer.  We call it an ``algorithmic framework'' rather than ``algorithm'' because some of its steps are underspecified.  Nevertheless, it serves as the basis for analyzing implementable algorithms.  After presenting the framework, we analyze its convergence rate using a potential function. Let $\x^*$ denote the minimizer of the composite function $F$ in \eqref{eq:composite_problem}.

	\begin{algorithm}[ht!]
          \textbf{Initialization:} Fix $\x_0 \in \R^n$. \\
\textbf{Set:} $\z_0 = \x_0$, and any affine subspace $\mathcal{M}_1 \supseteq \z_0 + \text{span} \{ G_{1/L}(\z_0) \}$\\
\textbf{for} $k = 1, 2, \ldots $
\begin{enumerate}
\item \textbf{Compute $\y_k$:} Let
\[ \y_k = \argmin_{\y} \{ \norm{\y-\x^*}^2 \, : \, \y \in \mathcal{M}_k\}. \]
\item \textbf{Compute $\x_k$:} Choose $\x_{k}$ such that
\begin{minipage}{0.85 \textwidth} \begin{equation} \label{eq:IA_x} \x_k = \argmin_{\x}
    \{ F(\x) \, : \, \x \in \mathcal{M}_k\}. \end{equation} \end{minipage}
\item \textbf{Compute $\z_k$:} Select an auxiliary $\z_k \in
  \mathcal{M}_k$ so that 
\begin{minipage}{0.85 \textwidth} \begin{equation} \label{eq:IA_z}  G_{1/\L}(\z_k)^T(\y_k-\z_k) \ge 0 \quad \text{and} \quad
F(\bar{\z}_k) \le F(\x_k) - \frac{1}{2\L} \norm{G_{1/\L}(\z_k)}^2,
\end{equation} \end{minipage}
\qquad where $\bar{\z}_k = \z_k - \frac{1}{\L}G_{1/\L}(\z_k)$.
\item \textbf{Update affine subspace:}
\[\mathcal{M}_{k+1} \supseteq \z_k + \text{span}\{\y_k-\z_k, G_{1/\L}(\z_k)\}\]
\end{enumerate}
\textbf{end}
		\caption{Idealized algorithmic framework for constraints (IA) }
		\label{alg:IA}
	\end{algorithm}

\begin{remark} \rm Steps 1 and 2 of Algorithm~\ref{alg:IA}, both unimplementable in general, are the same
as the corresponding steps in the idealized algorithm of \cite{KV}.  The new
ingredient in Algorithm 1 is Step 3 and the definition of the auxiliary sequence $\z_k$. Below we show there always exists an iterate $\z_k$ satisfying
    the conditions in Step 3 of Algorithm~\ref{alg:IA}. Both Nesterov's accelerated gradient method and
    geometric descent compute a point $\z_k$ and set $\x_k = \z_{k-1} -
    \tfrac{1}{\L} G_{1/\L}(\z_{k-1})$ whereas conjugate gradient directly
    computes $\x_k$ without knowledge of $\y_k$ and $\z_k$. 
    \end{remark}
    \begin{remark}
  \rm When the function $\Psi$ is an indicator
    of a closed convex set, neither iterate $\y_k$ nor $\z_k$
    is feasible in general.  
\end{remark}

We start by defining a potential at iteration $k \ge 1$ as follows
\[\Phi_k = \norm{\y_k-\x^*}^2 + \frac{ 2 (F(\x_k)-F(\x^*))}{\alpha}. \]
\begin{thm}[Convergence of IA] \label{thm:converge_IA} The
  iterates $\{\x_k, \z_k, \y_k\}_{k=1}^{\infty}$ generated by Algorithm~\ref{alg:IA} satisfy
  for $k = 1, 2, \ldots$
\[ \Phi_{k+1} \le \left (1-
    \sqrt{\frac{\a}{L}} \right ) \Phi_k \]
\end{thm}

\begin{proof} Define the points
\begin{align*}
\bar{\z}_k = \z_k - G_{1/\L}(\z_k) / \L \quad \text{and} \quad
  \bar{\bar{\z}}_k = \z_k - G_{1/L}(\z_k) / \a.
\end{align*}
The point $\bar{\z}_{k}$ satisfies both $F(\bar{\z}_k) \le F(\x_k) -
\norm{G_{1/\L}(\z_k)}^2/ (2\L)$ and $F(\x_{k+1}) \le
F(\bar{\z}_k)$ since $\bar{\z}_k \in \mathcal{M}_{k+1}$ for $k \ge 1$ by equations
\eqref{eq:IA_x} and \eqref{eq:IA_z}. This implies
\begin{equation} \begin{aligned} \label{eq:IA_decrease}
\frac{2(F(\x_{k+1})-F(\x^*))}{\a} +
\frac{\a}{\L} \cdot \frac{\norm{G_{1/\L}(\z_k)}^2}{\a^2}
& \le \frac{2(F(\bar{\z}_k)-F(\x^*))}{\a} +
\frac{\a}{L} \cdot \frac{\norm{G_{1/\L}(\z_k)}^2}{\a^2} \\
&\le \frac{2(F(\x_k)-F(\x^*))}{\a}.
\end{aligned} \end{equation}
Moreover, we observe from \eqref{eq:strong_cvx} with $\y = \x^*$ 
and $\x= \z_k$ that 
\begin{align} \label{eq:IA_bound_1}
\frac{-2 G_{1/\L}(\z_k)^T(\z_k-\x^*)}{\a} + \frac{1}{\a \L}
  \norm{G_{1/\L}(\z_k)}^2 + \norm{\z_k-\x^*}^2 \le \frac{-2( F(\bar{\z}_k)-F(\x^*))}{\a}.
\end{align}
We use this result to compute the following:
\begin{equation} \begin{aligned} \label{eq:rho}
\norm{\bar{\bar{\z}}_k - \x^*}^2 &= \norm{\bar{\bar{\z}}_k - \z_k}^2 +
  2(\bar{\bar{\z}}_k-\z_k)^T(\z_k-\x^*) + \norm{\z_k-\x^*}^2\\
&= \frac{\norm{G_{1/\L}(\z_k)}^2}{\a^2}
  -\frac{2G_{1/\L}(\z_k)^T(\z_k-\x^*)}{\a} + \norm{\z_k-\x^*}^2\\
&\le \left (1- \frac{\a }{\L } \right )
  \frac{\norm{G_{1/\L}(\z_k)}^2}{\alpha^2} - \frac{2(F(\bar{\z}_k)-F(\x^*))}{\a}\\
&\le \left (1 - \frac{\a}{\L } \right )
  \frac{\norm{G_{1/\L}(\z_k)}^2}{\a^2} - \frac{2(F(\x_{k+1})-F(\x^*))}{\a} =: \rho_k^2,
\end{aligned}
\end{equation}
where the first inequality follows from \eqref{eq:IA_bound_1} and the last inequality is a consequence of the assumption
$F(\x_{k+1}) \le F(\bar{\z}_k)$ as $\bar{\z}_k \in \mathcal{M}_{k+1}$. The radius $\rho_k$ will be used in
Corollary~\ref{cor:intersection}. Thus, we have $\x^*\in
B(\bar{\bar{\z}}_k, \rho_k)$. Next, we see from \eqref{eq:IA_decrease}
for $k \ge 1$
\begin{equation} \begin{aligned} \label{eq: IA_y_bound}
&\norm{\y_k-\x^*}^2 = \norm{\y_k-\x^*}^2 + \frac{ 2(F(\x_k)-F(\x^*))}{\a} -
  \frac{2 (F(\x_k)-F(\x^*))}{\a} \\
& \qquad \le \norm{\y_k-\x^*}^2 + \frac{ 2(F(\x_k)-F(\x^*))}{\a} - \frac{\a}{\L}
  \cdot \frac{\norm{G_{1/\L}(\z_k)}^2}{\a^2} - \frac{2
  (F(\x_{k+1})-F(\x^*))}{\a} =: \sigma_k^2.
\end{aligned} \end{equation}
The expressions for \eqref{eq:rho} and \eqref{eq: IA_y_bound} share
many quantities like those in Corollary~\ref{cor:intersection}. For Corollary~\ref{cor:intersection}, we need to identify $r_1$,
$r_2$, $C$, and $\varepsilon$ from $\rho_k$ and $\sigma_k$. As such,
we see that we have the following relationships for $k \ge 1$
\[\rho_k^2 = (1-\varepsilon) r_2^2 - C \quad \text{and} \quad \sigma_k^2
  = r_1^2 - \varepsilon r_2^2 - C\]
\[\text{where} \quad r_2^2 = \frac{\norm{G_{1/\L}(\z_k)}^2}{\a^2}, \quad r_1^2 =
  \Phi_k, \quad \varepsilon = \frac{\a}{\L}, \quad \text{and} \quad C = \frac{2(F(\x_{k+1})-F(\x^*))}{\a}.\]
We have now defined $r_1, r_2,C,$ and $\varepsilon$ for
Corollary~\ref{cor:intersection} with $\y = \y_k$ and $\z = \bar{\bar{\z}}_k$. First we confirm that
$\norm{\y_k-\bar{\bar{\z}}_k}^2 \ge r_2^2$. The choice of $\z_k$ ensures
that $G_{1/\L}(\z_k)^T(\y_k-\z_k) \ge 0$ for all $k \ge 1$. 
It follows from simple computations
\begin{equation} \begin{aligned} \label{eq:IA_1}
\norm{\y_k-\bar{\bar{\z}}_k}^2 &= \norm{\y_k - \z_k + \z_k -
  \bar{\bar{\z}}_k}^2\\
&= \norm{\y_k-\z_k + \tfrac{1}{\a} G_{1/\L}(\z_k)}^2\\
&= \norm{\y_k-\z_k}^2 + \tfrac{1}{\a^2} \norm{G_{1/\L}(\z_k)}^2 +
  \tfrac{2}{\a} G_{1/\L}(\z_k)^T(\y_k-\z_k)\\
&\ge \tfrac{1}{\a^2} \norm{G_{1/\L}(\z_k)}^2 = r_2^2.
\end{aligned} \end{equation}
Moreover, we have $\rho_k + \sigma_k \ge \norm{\bar{\bar{\z}}_k-\x^*} +
\norm{\y_k-\x^*} \ge \norm{\bar{\bar{\z}}_k-\y_k}$. Therefore by
Corollary~\ref{cor:intersection} there exists a $\z_k^* \in \text{aff}\{\y_k, \bar{\bar{\z}}_k\}$ such that $B(\y_k, \sigma_k)
\cap B(\bar{\bar{\z}}_k, \rho_k) \subset B\left (\z_k^*,
\left(1-\sqrt{\tfrac{\a}{\L}}\right) \Phi_k - \tfrac{2 (F(\x_{k+1})-F(\x^*))}{\a}
\right )$. Because $\x^* \in B(\y_k, \sigma_k)$ and $\x^* \in B(\bar{\bar{\z}}_k,
\rho_k)$, we have 
\[ \norm{\z_k^*-\x^*}^2 \le \left(1- \sqrt{\frac{\a}{\L}} \right) \Phi_k - \frac{2
    (F(\x_{k+1})-F(\x^*))}{\a}.\]
Since $\y_{k+1}$ is the optimizer of $\norm{\y-\x^*}$ over $\y \in
\mathcal{M}_{k+1}$ and $\z_k^* \in \mathcal{M}_{k+1}$, then $\y_{k+1}$
is at least as close to $\x^*$ as $\z_k^*$. Hence, we conclude
\[ \norm{\y_{k+1}-\x^*}^2 \le \norm{\z_k^*-\x^*}^2 \le \left(1- \sqrt{\frac{\a}{\L}} \right) \Phi_k - \frac{2
    (F(\x_{k+1})-F(\x^*))}{\a}. \]
The result follows by adding $\frac{2(F(\x_{k+1})-F(\x^*))}{\a}$ to both
sides. \end{proof}

In Step 3 of Algorithm~\ref{alg:IA}, we need to construct an iterate $\z_k
\in \mathcal{M}_k$ such that the following hold
\begin{equation} \label{eq: aux_iter}  G_{1/\L}(\z_k)^T(\y_k-\z_k) \ge 0
  \quad \text{and} \quad F(\bar{\z}_k) \le F(\x_k) - \frac{1}{2\L}
  \norm{G_{1/\L}(\z_k)}^2. \end{equation}
These two properties were a crucial part of Chen et al.\ \cite{GD_chen} proof
of geometric descent for convex composites. We restate their results
and provide a proof for completeness. 

\begin{lem}[Lemma 3.2 of \cite{GD_chen}] \label{lem:add_bound}
Assume $F=f+\Psi$ in which $f$ is $\L$-smooth and convex and $\Psi$ is closed, proper and convex.  Condition \eqref{eq: aux_iter} holds
  if $\z_k \in \mathcal{M}_k$ satisfies
\[G_{1/\L}(\z_k)^T(\y_k-\z_k) \ge 0 \quad \text{and} \quad
  -G_{1/\L}(\z_k)^T(\x_k-\z_k) \le 0.\]
\end{lem} 

\begin{proof} By Lemma~\ref{lem:smooth} with $\x = \z_k$ and $\y =
  \x_k$, we have 
\begin{align*}
F(\bar{\z}_k) \le F(\x_k) - G_{1/\L}(\z_k)^T(\x_k-\z_k)
  - \frac{1}{2\L} \norm{G_{1/\L}(\z_k)}^2.
\end{align*}
We apply the second assumption to complete the proof. 
\end{proof}
For notational simplicity, we drop the subscripts $k$, allowing $\x,\y$ to be any
two points in $\R^n$, and we seek a solution $\z$ to the following inequalities:
\begin{equation} \label{eq: condition_z_k} G_{1/\L}(\z)^T(\y-\z) \ge 0 \quad \text{and} \quad
  -G_{1/\L}(\z)^T(\x-\z) \le 0. \end{equation} As in Chen et al.\  \cite{GD_chen}, we define the following functions for any
given $\x, \y \in \R^n$ $(\x \neq \y)$, 
\begin{equation} \label{eq:increase_funct}
h_{\x,\y}(\z) = G_{1/\L}(\z)^T(\y-\x),  \, \, \forall \z \in \R^n \quad
\text{and} \quad \bar{h}_{\x,\y}(s) = h_{\x,\y}(\x + s(\y-\x)),
\, \, \forall s \in \R.
\end{equation}
The following lemma illustrates how to generate an iterate satisfying
the two conditions \eqref{eq: condition_z_k} in
Lemma~\ref{lem:add_bound} and hence the conditions \eqref{eq: aux_iter}.
\begin{lem}[Lemma 3.4 and 3.3 in \cite{GD_chen}] \label{lem: exist_z} Assume $F=f+\Psi$ where $f$ is $\L$-smooth (but not necessarily convex) and $\Psi$ is closed, proper and convex.  Let $\x, \y$ be distinct points in
  $\R^n$. There exists a point $\z \in \text{\rm aff}\{\x,\y\}$ of the following form satisfying the conditions in \eqref{eq:
    condition_z_k}:
\begin{equation} \label{eq:z_rep} \z  = \begin{cases}
\y, & \text{if $\bar{h}_{\x, \y}(1) \le 0$}\\
\x, & \text{if $\bar{h}_{\x,\y}(0) \ge 0$}\\
\mbox{$\x + s(\y-\x)$ for some $s\in[0,1]$}, &\text{otherwise.}
\end{cases} \end{equation}
\end{lem}

\begin{proof} First, we will show that the function $h_{\x,\y}(\cdot)$
  is $3\L$-Lipschitz continuous. For any closed, convex function $h :
  \R^n \to \R$,
  the proximal operator is nonexpansive (see \textit{e.g.} \cite{BC_book}) 
\[ \norm{\prox_h(\x) - \prox_h(\y)} \le \norm{\x-\y}, \quad \text{for all
    $\x, \y \in \R^n$}.\]
For any $\z_1, \z_2 \in \R^n$, we have the following
\begin{align*}
|h_{\x,\y}(\z_1) - h_{\x,\y}(\z_2)| &= |
  (G_{1/\L}(\z_1)-G_{1/\L}(\z_2))^T(\y-\x)|\\
&\le \L \norm{\y-\x} \biggl( \norm{\z_1-\z_2} \\
&\hphantom{\le} \quad\mbox{}+ \norm{\prox_{\Psi/\L}(\z_1- \nabla f(\z_1))/\L - \prox_{\Psi/L}(\z_2-\nabla f(\z_2)/\L)} \biggr )\\
&\le \L \norm{\y-\x} \left ( \norm{\z_1-\z_2} + \norm{\z_1 - \z_2} +
  \frac{1}{\L} \norm{\nabla f(\z_1)-\nabla f(\z_2)}  \right ).
\end{align*}
Because $\nabla f(\cdot)$ is $L$-Lipschitz continuous, we obtain the
desired result.  

Next, it is clear that the iterate $\z$ defined in \eqref{eq:z_rep}
lives in the affine space defined by $\x$ and $\y$. A quick computation shows
that when $\bar{h}_{\x,\y}(1) \le 0$ or $\bar{h}_{\x,\y}(0) \ge 0$ and
$\z = \y \, (\text{resp. } \x)$ then \eqref{eq: condition_z_k} holds. Suppose
$\bar{h}_{\x,\y}(1) > 0$ and $\bar{h}_{\x,\y}(0) < 0$. The intermediate
value theorem implies there exists an $s \in [0,1]$ such that
$G_{1/\L}(\x+s(\y-\x))^T(\y-\x) = \bar{h}_{\x,\y}(s) =0$. If $\z = \x+s(\y-\x)$ then
\[G_{1/\L}(\z)^T(\y-\z) = (1-s) G_{1/\L}(\x + s(\y-\x))^T(\y-\x) = 0\]
\[\text{and} \qquad -G_{1/\L}(\z)^T(\x-\z) = s G_{1/\L}(\x +
  s(\y-\x))^T(\y-\x) = 0.\]
The result follows.  
\end{proof}

\section{Geometric descent in the composite setting for strongly convex functions}\label{sec:chen}
In this section, we present the geometric descent algorithm for the
composite setting as established in Chen et al.\ \cite{GD_chen} and our potential-based analysis of it. The proof of the
idealized algorithm follows the ideas of geometric descent (GD) as
presented in both Bubeck et al.\  \cite{GD_Bubeck} and Chen et al.\  \cite{GD_chen}. Unlike
IA, the potential function for geometric descent must be
computable. The idea is to first set $\x_k = \bar{\z}_k$ and then upper bound
$\frac{2(F(\bar{\z}_k)-F(\x^*))}{\a} + \norm{\y_k-\x^*}^2$ by a computable
quantity. We present a modified version of Chen et al. geometric descent in Algorithm~\ref{alg:geometric_descent}.

	\begin{algorithm}[ht!]
          \textbf{Initialization:} Fix $\x_0 \in \R^n$\\
\textbf{Set:} Iterates $\z_0 :=\x_0$ and
          $\y_1 := \x_0 - \tfrac{1}{\a} G_{1/\L}(\x_0)$, and fix the radius
          $\tilde{\xi}_1^2 := \left ( \frac{1}{\a^2} - \frac{1}{\L\a}
          \right ) \norm{G_{1/\L}(\x_0)}^2$.\\
\textbf{for} $k = 1, 2, \ldots$
\begin{enumerate}
\item Set the iterates 
\[ \bar{\z}_{k-1} := \z_{k-1} - \frac{G_{1/\L}(\z_{k-1})}{\L} \quad
  \text{and} \quad \bar{\bar{\z}}_{k-1} = \z_{k-1}- \frac{G_{1/\L}(z_{k-1})}{\a}. \]
\item \textbf{Compute $\x_k$:} $\x_{k} := \bar{\z}_{k-1}$
\item \textbf{Compute $\z_k$:} Find a $\z_k \in \text{aff} \{
  \bar{\z}_{k-1}, \y_k\}$ such that 
\begin{minipage}{0.85 \textwidth}\begin{equation} \label{eq: gd_z} F(\bar{\z}_k) \le F(\bar{\z}_{k-1}) - \frac{1}{2\L}
  \norm{G_{1/\L}(\z_k)}^2 \quad \text{and} \quad
  G_{1/\L}(\z_k)^T(\y_k-\z_k) \ge 0. \end{equation} \end{minipage}
\item \textbf{Compute $\y_k$:} Determine $\lambda_{k+1}$ and
  $\tilde{\xi}_{k+1}$ as in \eqref{eq:GD_2} and \eqref{eq:GD_4} and set 
\begin{minipage}{0.85 \textwidth} \begin{equation} \label{eq:GD_y} \y_{k+1} := (1-\lambda_{k+1}) \bar{\bar{\z}}_{k} + \lambda_{k+1}
  \y_k. \end{equation} \end{minipage}
\end{enumerate}
\textbf{end}
		\caption{Geometric descent for the composite setting}
		\label{alg:geometric_descent}
	\end{algorithm}
	\bigskip

As in the analysis of IA, we have that \eqref{eq:rho} holds for all $k
\ge 1$ 
\begin{align} \label{eq:gd_2}
\norm{\bar{\bar{\z}}_k-\x^*}^2 \le \left ( 1 -
  \frac{\a}{L} \right ) \frac{\norm{G_{1/\L}(\z_k)}^2}{\a^2} - \frac{2
  (F(\bar{\z}_k)-F(\x^*))}{\a} =: \rho_k^2 =: \tilde{\rho}_k^2 - \gamma_k,
\end{align}
where $\tilde{\rho}_k^2 = \left ( \tfrac{1}{\a^2} -
  \tfrac{1}{\L\a} \right ) \norm{G_{1/\L}(\z_k)}^2$ and $\gamma_k =
\tfrac{2(F(\bar{\z}_k)-F(\x^*))}{\a}$. We now define a sequence of
positive scalars $\{\tilde{\xi}_k\}_{k =1}^\infty$ and their corresponding
$\{\sigma_k\}_{k=1}^\infty$ by the following
relation 
\begin{equation} \label{eq:GD_sigma_def} \sigma_k^2 := 
\tilde{\xi}_k^2 - \frac{\a}{\L} \frac{\norm{G_{1/\L}(\z_k)}^2}{\a^2}
-\gamma_k. \end{equation}
Note, for this discussion, the quantities $\sigma_k,\rho_k,\xi_k$ represent noncomputable radii (i.e. they require prior knowledge of $\x^*$), while  $\tilde\sigma_k,\tilde\rho_k,\tilde\xi_k$ represent related quantities that are computed by the algorithm.
The sequence $\{\sigma_k\}_{k=1}^\infty$ will upper bound the distance from $\y_k$ to the optimizer. For $k
= 1$, we have that 
\[\tilde{\xi}_1^2 = \left (1-\frac{\a}{\L} \right )
  \frac{\norm{G_{1/\L}(\z_0)}^2}{\a^2} \, \, \text{and} \, \,  \sigma_1^2 =  \left (1-\frac{\a}{\L} \right )
  \frac{\norm{G_{1/\L}(\z_0)}^2}{\a^2} - \frac{\a}{\L}
  \frac{\norm{G_{1/\L}(\z_1)}^2}{\a^2} - \frac{2(F(\bar{\z}_1)-
    F(\x^*))}{\a}.
\]
Inductively, we define $\tilde{\xi}_{k+1}$ for any $k \ge 1$ by 
\begin{align}
&\textbf{Set} \quad \tilde{\rho}_k^2 := \left ( 1 - \frac{\a}{\L }
  \right ) \frac{\norm{G_{1/\L}(\z_k)}^2}{\a^2}, \quad \tilde{\sigma}_k^2 =
  \tilde{\xi}_k^2 - \frac{\a}{\L} \frac{\norm{G_{1/\L}(\z_k)}^2}{\a^2},
  \quad \text{and} \quad
  \delta_k^2 = \frac{\norm{G_{1/\L}(\z_k)}^2}{\a^2} \label{eq:GD_bound_3}\\
&\textbf{if} \quad \tilde{\sigma}_k^2 \le \tilde{\rho}_k^2 + \delta_k^2 \label{eq:GD_if_1} \\
& \qquad \text{Compute} \quad \lambda_{k+1} = \frac{\delta_k^2 +
  \tilde{\rho}_k^2 - \tilde{\sigma}_k^2}{2 \delta_k^2}
  \quad \text{and} \quad \tilde{\xi}_{k+1}^2 =
  \frac{\tilde{\rho}_k^2}{2} + \frac{\tilde{\sigma}_k^2}{2}
  - \frac{\delta_k^2}{4} - \frac{(\tilde{\rho}_{k}^2-
  \tilde{\sigma}_{k}^2)^2}{4\delta_{k}^2} \label{eq:GD_2}\\
&\textbf{else} \nonumber \\
& \qquad \text{Set} \quad \lambda_{k+1} = 0 \quad \text{and} \quad
  \tilde{\xi}_{k+1}^2 = \tilde{\rho}_k^2.
  \label{eq:GD_4}
\end{align}
It remains to check that $\tilde{\xi}_k^2$ and $\tilde{\sigma}_k^2$ in both cases are positive.

\begin{thm} For $k = 1, 2, \ldots$, the following hold:
\begin{equation} \label{gd: bound_xi_sigma} \tilde{\xi}_k^2 \ge \norm{\y_k-\x^*}^2 +
  \frac{2(F(\bar{\z}_{k-1})-F(\x^*))}{\a} \quad \text{and} \quad \tilde{\sigma}_k^2 \ge \norm{\y_k-\x^*}^2 +
  \frac{2(F(\bar{\z}_k)-F(\x^*))}{\alpha}. \end{equation}
\end{thm}

\begin{proof} First, the inequalities in \eqref{gd: bound_xi_sigma} can be rewritten as
  $\tilde{\xi}_k^2 \ge \norm{\y_k-\x^*}^2 + \gamma_{k-1}$ and
  $\tilde{\sigma}_k^2 \ge \norm{\y_k-\x^*}^2 + \gamma_k$. Moreover by
  \eqref{eq: gd_z}, we have 
\[ \frac{2(F(\bar{\z}_k)-F(\x^*))}{\a} + \frac{\a}{\L}
  \frac{\norm{G_{1/\L}(\z_k)}^2}{\a^2} \le
  \frac{2(F(\bar{\z}_{k-1})-F(\x^*))}{\a}.\]
Therefore if the first inequality of \eqref{gd: bound_xi_sigma} holds, then the second also holds due to \eqref{eq:GD_bound_3}. We prove this result for $\tilde{\xi}_k$ inductively. For $k = 1$, we have from \eqref{eq:gd_2} and the initialization in Algorithm~\ref{alg:geometric_descent},
\begin{equation} \label{eq: gd_1} \norm{\y_1-\x^*}^2 =  \norm{\bar{\bar{\z}}_0 -\x^*}^2 \le \left ( 1- \frac{\a}{\L}
  \right ) \frac{\norm{G_{1/\L}(\z_0)}^2}{\a^2} - \frac{2(F(\bar{\z}_0) -
  F(\x^*))}{\a} = \tilde{\xi}_1^2 - \gamma_0.\end{equation}
The result follows for $k = 1$. 


Assume $k \ge 1$ and the induction hypothesis
$\tilde{\sigma}_k^2 \ge \norm{\y_k-\x^*}^2 +
\gamma_k$ and $\tilde{\xi}_k^2 \ge \norm{\y_k-\x^*}^2 +
\gamma_{k-1}$ holds. The idea is that $\sigma_k$ and $\rho_k$ take the specific form given in
Corollary~\ref{cor:intersection}. Recall the definitions of $\rho_k$
and $\sigma_k$ given in \eqref{eq:gd_2} and \eqref{eq:GD_sigma_def} resp., namely for $k \ge 1$, the radii satisfy
\[ \sigma_k^2 = \tilde{\xi}_k^2 - \frac{\a}{\L}
  \frac{\norm{G_{1/\L}(\z_k)}^2}{\a^2} - \gamma_k \quad \text{and}
  \quad \rho_k^2 = \left ( 1- \frac{\a}{\L} \right )
  \frac{\norm{G_{1/\L}(\z_k)}^2}{\a^2} - \gamma_{k}. \]
Hence, we can write $\sigma_k$ and $\rho_k$ as
\begin{equation} \label{eq:GD_9} \sigma_k^2 = r_1^2 - \varepsilon r_2^2 - C \quad \text{and} \quad
  \rho_k^2 = (1-\varepsilon) r_2^2 - C, \end{equation}
where $r_1^2 = \tilde{\xi}_k^2$, $r_2^2 =
\frac{\norm{G_{1/\L}(\z_k)}^2}{\a^2}$, $\varepsilon = \frac{\a}{\L}$, $\z
= \bar{\bar{\z}}_k$, $\y = \y_k$, 
and $C = \gamma_k$. We have identified the terms in
Corollary~\ref{cor:intersection}; it remains to show that we can apply
this result. From the induction hypothesis and the definition of
$\tilde{\rho}_{k}$ in \eqref{eq:gd_2}, we observe the following 
\begin{equation} \label{eq:gd_4} 
  \norm{\y_{k}-\bar{\bar{\z}}_{k}}^2 \le \norm{\y_{k}-\x^*} + \norm{\bar{\bar{\z}}_{k}-\x^*} \le
\sqrt{r_1^2 - \varepsilon r_2^2 - C} +
\sqrt{(1-\varepsilon) r_2^2 - C}. \end{equation} 
As in \eqref{eq:IA_1} of the IA analysis, we also have by \eqref{eq: gd_z},
\begin{equation} \label{eq:gd_5} r_2^2 = \frac{\norm{G_{1/\L}(\z_k)}^2}{\a^2}
  \le \norm{\y_k-\bar{\bar{\z}}_k}^2. 
\end{equation}  
We show $\y_{k+1}$ in GD is defined exactly as $\z^*$ in
Corollary~\ref{cor:intersection}. Suppose \eqref{eq:GD_if_1} holds,
which under the identification of \eqref{eq:GD_9}, agrees with $r_1^2
\le 2 r_2^2$. A quick calculation shows 
\[\lambda_{k+1} = \frac{\delta_k^2+\tilde{\rho}_k^2 -
    \tilde{\sigma}_k^2}{2 \delta_k^2} = \frac{2r_2^2-r_1^2}{2r_2^2}; \] 
thus $\lambda_{k+1}$ in \eqref{eq:GD_2} equals $\lambda$ in
\eqref{eq:value_z}. In particular from \eqref{eq:value_xi} in
Corollary~\ref{cor:intersection}, the iterate $\y_{k+1}$ satisfies
$B(\y_k, \sigma_k) \cap B(\bar{\bar{\z}}_k, \rho_k) \subset B(\y_{k+1}, \xi)$
where $\xi$ is defined in \eqref{eq:value_xi} by
\[ \xi^2 = r_1^2 -\varepsilon r_2^2 - \frac{r_1^4}{4r_2^2} - C. \]
A quick computation shows 
\begin{equation} \label{eq:GD_blah} \xi^2 = r_1^2-\varepsilon r_2^2 -
  \frac{r_1^4}{4r_2^2} - C = \frac{\tilde{\rho}_k}{2} +
  \frac{\tilde{\sigma}_k^2}{2} - \frac{\delta_k^2}{4} - \frac{
    (\tilde{\rho}_k^2 - \tilde{\sigma}_k)^2}{\delta_k^2} - \gamma_k = \tilde{\xi}_{k+1}^2 - \gamma_k, \end{equation}
where the second equality comes from the identifications in
\eqref{eq:GD_9} and the last equality from
\eqref{eq:GD_2}. 
Since the point $\x^*$ lies in $B(\y_k, \sigma_k)$ and
$B(\bar{\bar{\z}}_k, \rho_k)$, we obtain that
\[ \norm{\y_{k+1} - \x^*}^2 \le \xi^2 = \tilde{\xi}_{k+1}^2 - \gamma_k. \]

If the condition
\eqref{eq:GD_if_1} fails, then $\lambda_{k+1} = 0$ and $\y_{k+1} =
\bar{\bar{\z}}_{k}$. By definition of $\rho_{k}$, we have
$\norm{\y_{k+1}-\x^*}^2 \le \tilde{\rho}_{k}^2 - \gamma_{k}$, which by
\eqref{eq:GD_4} implies $\norm{\y_{k+1}-\x^*}^2 \le \tilde{\xi}_{k+1}^2 -
\gamma_k$.
\end{proof}

\begin{thm}[Convergence of geometric descent] For $k = 1, 2, \ldots$,
  we have
\[\tilde{\xi}_{k+1}^2 \le \left ( 1- \sqrt{\frac{\a}{\L}} \right ) \tilde{\xi}_{k}^2.\]
\end{thm}

\begin{proof} As in \eqref{eq:GD_9}, the following identifications of
  $\rho_k$ and $\sigma_k$ hold for all $k \ge 1$
\begin{align*}
\sigma_k^2 = r_1^2 -
  \varepsilon r_2^2 - C \quad \text{and} \quad  \rho_k^2 = (1-\varepsilon) r_2^2
  -C,
\end{align*} 
where $r_1^2 = \tilde{\xi}_k^2$, $r_2^2 = \frac{\a}{\L}
\frac{\norm{G_{1/\L}(\z_k)}^2}{\a^2}$, $\varepsilon =\frac{\a}{\L}$, and
$C = \gamma_k$.
We again take two cases depending on whether the
  condition \eqref{eq:GD_if_1} holds. If condition \eqref{eq:GD_if_1}
  holds, then we have by \eqref{eq:GD_blah}
\[ \tilde{\xi}_{k+1}^2 - \gamma_k = r_1^2-\varepsilon r_2^2 - \frac{r_1^4}{4r_2^2} - C.\]
In \eqref{eq:value_xi} of Corollary~\ref{cor:intersection}, we showed 
\[ r_1^2-\varepsilon r_2^2 - \frac{r_1^4}{4r_2^2} - C \le
  (1-\sqrt{\varepsilon}) r_1^2-C = \left (1 -\sqrt{\frac{\a}{\L}} \right
    )\tilde{\xi}_k^2-\gamma_k, \]
so the result follows. 

If condition \eqref{eq:GD_if_1} does not hold, namely $r_1^2 \ge
2r_2^2$, then one has
\[\tilde{\xi}_{k+1}^2 = (1-\varepsilon) r_2^2 \le
\frac{(1-\varepsilon)}{2} r_1^2 = \frac{(1-\varepsilon)}{2}
\tilde{\xi}_k^2.\]
A simple computation shows that $1/2 - \sqrt{\varepsilon} +
\tfrac{\varepsilon}{2} = \tfrac{1}{2} (\sqrt{\varepsilon}-1)^2 \ge
0$. This implies that $\tfrac{1}{2} (1-\varepsilon) \le
(1-\sqrt{\varepsilon})$. Hence, the result follows.
\end{proof}

\section{Accelerated Gradient}\label{nest:analysis}
In this section, we show that Nesterov's first accelerated gradient
method for composite functions \cite{intro_lect} can also be viewed as an approximation to the idealized algorithmic framework and can be analyzed with the same potential.  For this section, we assume $f$ is strongly convex, and let $\kappa=\L/\a$.  This is sometimes called the ``condition number'' of $f$.  Nesterov's algorithm is presented below as Algorithm~\ref{alg:AG}.

	\begin{algorithm}[ht!]
	\label{alg:AG}
          \textbf{Initialization:} Choose $\x_0 \in \R^n$\\
\textbf{Set:} $\w_0 = \x_0$\\
\textbf{for }$k = 1, 2, \ldots $
\begin{enumerate}
\item \textbf{Compute $\x_k$:} 
\begin{minipage}{0.85 \textwidth} \begin{equation} \label{eq:AG_x} \x_k
    := \prox_{\Psi/\L}\left ( \w_{k-1}- \frac{1}{\L} \nabla f(\w_{k-1}) \right ) \end{equation} \end{minipage}
\item \textbf{Compute $\w_k$:} 
\begin{minipage}{0.85 \textwidth} \begin{equation} \label{eq:AG_w} \w_k := \x_k + \theta (\x_k-\x_{k-1}),
  \quad \text{where} \quad \theta =
  \tfrac{\sqrt{\kappa}-1}{\sqrt{\kappa} + 1}. \end{equation} \end{minipage}
\end{enumerate}
		\caption{Accelerated Gradient (AG) }
	\end{algorithm}
	\bigskip

For the purpose of the convergence analysis, define the following auxiliary sequences
of vectors and scalars:
\begin{align}
\bar{\bar{\w}}_k &= \w_k - \tfrac{G_{1/\L}(\w_k)}{\a}, \label{eq:AG_wbarbar} \\
\y_0 &= \x_0, \label{eq:AG_y_initial}\\
\y_k &= \x_k + \tau(\x_k-\x_{k-1}) \label{eq:AG_y}\\
\tilde{\sigma}_0^2 &= \norm{\y_0-\x^*}^2 + \frac{2(F(\x_0)-F(\x^*))}{\a}\label{eq:ag_sigma0}\\
\tilde{\sigma}_{k+1}^2 & = (1-\kappa^{-1/2}) \tilde{\sigma}_k^2 -
                     (\kappa^{1/2} - \kappa^{-1/2}) \norm{\w_k-\x_k}^2 \label{eq:ag_sigmak1}\\
\tau &= \sqrt{\kappa} -1
\end{align}
Unlike the case of Geometric Descent, these auxiliary scalars and vectors play no role in the computation of the iterate.  However, with the exception of the initialization of $\tilde\sigma_0$, all are computable and could be used in an adaptive or hybrid version of AG as in \cite{KV}.  Furthermore, the equation \eqref{eq:ag_sigma0} defining $\tilde\sigma_0$ can be replaced by a computable upper bound using the result of \eqref{gd: bound_xi_sigma}.

It is clear from \eqref{eq:ag_sigmak1} that for $k = 0, 1, 2,\ldots$ we have
\[ \tilde{\sigma}_{k+1}^2 \le \left(1-\frac{1}{\sqrt{\kappa}} \right)
  \tilde{\sigma}_k^2. \]
\begin{lem} \label{lem:nesterov_affine} Fix $\x_0 \in \R^n$. Suppose $\{\x_k, \w_k\}_{k=0,1,\ldots}$ are
  iterates generated by Algorithm~\ref{alg:AG}. Then $\y_{k+1}$ defined in
  \eqref{eq:AG_y} with $\y_0 = \x_0$ is a convex combination of
  $\bar{\bar{\w}}_k$ and $\y_k$. In particular for all $k = 0, 1,
  \ldots$, we have
\begin{equation} \label{eq:nesterov_y} \y_{k+1} = \frac{1}{\sqrt{\kappa}} \bar{\bar{\w}}_k +
  \left (1-\frac{1}{\sqrt{\kappa}} \right ) \y_k. \end{equation}
\end{lem}

\begin{proof} First, we will show that for all $k \ge 0$, the following equality
\begin{equation} \label{eq:AG_y_other} \y_k = (\sqrt{\kappa} + 1) \w_k - \sqrt{\kappa} \cdot \x_k. \end{equation}
The initialization of \eqref{eq:AG_y_initial} and $\w_0 = \x_0$ proves
the result for $k = 0$. Suppose $k \ge 1$. Using equations
\eqref{eq:AG_y} and \eqref{eq:AG_w}, we have
\begin{align*}
\y_k = \x_k + \frac{\tau}{\theta} (\w_k - \x_k) = \x_k + (\sqrt{\kappa} +
  1)(\w_k-\x_k) = (\sqrt{\kappa} + 1) \w_k - \sqrt{\kappa} \x_k.
\end{align*} 
Starting from \eqref{eq:AG_y}, we show $\y_{k+1} \in \text{aff} \{\y_k, \bar{\bar{\w}}_k\}$ as follows
\begin{align*}
\y_{k+1} &= \sqrt{\kappa} \cdot \x_{k+1} - (\sqrt{\kappa} -1 ) \x_k\\
&= \sqrt{\kappa} \left (\w_k - \tfrac{1}{\alpha\kappa}
  G_{1/\L}(\w_k) \right ) - (\sqrt{\kappa} -1 ) \x_k\\
&= \frac{1}{\sqrt{\kappa}} \left ( \w_k - \tfrac{1}{\a}
  G_{1/\L}(\w_k) \right ) -
  (\sqrt{\kappa}-1) \x_k + \left  (\sqrt{\kappa} -
  \tfrac{1}{\sqrt{\kappa}} \right ) \w_k\\
&= \tfrac{1}{\sqrt{\kappa}} \bar{\bar{\w}}_k + \left (1-
  \tfrac{1}{\sqrt{\kappa}} \right ) \left ( (\sqrt{\kappa} +1) \w_k -
  \sqrt{\kappa} \x_k \right )\\
&= \tfrac{1}{\sqrt{\kappa}} \bar{\bar{\w}}_k + \left (1 -
  \tfrac{1}{\sqrt{\kappa}} \right ) \y_k.
\end{align*}
\end{proof}

\begin{lem} Fix $\x_0 \in \R^n$. Suppose $\{\x_k, \w_k\}_{k\ge 0}$ are
  iterates generated by Algorithm~\ref{alg:AG} and let the iterates
  $\{\y_k\}_{k \ge 0}$ be defined as in
  \eqref{eq:AG_y} with $\y_0 = \x_0$. Then for each $k = 0,1, \ldots$,
  we have 
\begin{align} \label{eq:CG_recurr}
\norm{\y_k-\x^*}^2 + \frac{2 (F(\x_k)-F(\x^*) )}{\alpha} \le \tilde{\sigma}_k^2.
\end{align}
\end{lem}

\begin{proof} We prove \eqref{eq:CG_recurr} by induction. For the case
  $k =0$, the results follows from \eqref{eq:ag_sigma0}. We assume
  that \eqref{eq:CG_recurr} holds. We take $\z$ and $\y$ appearing in
  Lemma~\ref{lem:intersection_ball_ideal} to be $\bar{\bar{\w}}_k$ and
  $\y_k$ respectively. Next we define $\delta, \rho, \sigma$ to be used
  in Lemma~\ref{lem:intersection_ball_ideal}. In the case of $\rho$,
  we copy the definition used in the analysis of IA:
\begin{align*}
\norm{\bar{\bar{\w}}_k-\x^*}^2 &\le \left ( 1 -
  \frac{\a}{\L} \right ) \frac{\norm{G_{1/\L}(\w_k)}^2}{\a^2} - \frac{2
  (F(\bar{\w}_k)-F(\x^*))}{\a}\\
&= \left ( 1 - \frac{\a}{\L } \right )
  \frac{\norm{G_{1/\L}(\w_k)}^2}{\a^2} - \frac{2( F(\x_{k+1})-F(\x^*))}{\a} =: \rho^2.
\end{align*}
We use the induction hypothesis to define $\sigma$:
\begin{align*}
\norm{\y_k-\x^*}^2 \le \tilde{\sigma}_k^2 -
  \frac{2(F(\x_k)-F(\x^*))}{\a} =: \sigma^2.
\end{align*}
For the radius $\delta$, first, we observe using Lemma~\ref{lem:strongly_convex} with $\y = \x_k$, $\x = \w_k$, and $\bar{\x} = \x_{k+1}$ the following inequality holds
\begin{equation}
\begin{aligned} \label{eq:Nesterov_1}
\frac{-2\sqrt{\kappa}}{\alpha} G_{1/\L}(\w_k)^T&(\x_k-\w_k)\\
& \ge \frac{2
  \sqrt{\kappa} ( F(\x_{k+1})-F(\x_k)) }{\a} + \frac{\sqrt{\kappa}}{
  \a \L}
  \norm{G_{1/\L}(\w_k)}^2 + \sqrt{\kappa} \norm{\x_k-\w_k}^2.
\end{aligned}
\end{equation}
We construct a choice for $\delta$ starting from equations
\eqref{eq:AG_wbarbar} and  \eqref{eq:AG_y_other}:
\begin{align*}
\norm{\bar{\bar{\w}}_k-\y_k}^2 &= \norm{\sqrt{\kappa}(\x_k-\w_k) -
                               \tfrac{G_{1/\L}(\w_k)}{\a}}^2\\
&= \kappa \norm{\x_k-\w_k}^2 + \frac{1}{\a^2} \norm{G_{1/\L}(\w_k)}^2
  - \frac{2\sqrt{\kappa}}{\a} G_{1/\L}(\w_k)^T(\x_k-\w_k)\\
&\ge (\kappa + \sqrt{\kappa}) \norm{\x_k-\w_k}^2 + \left (
  1 + \frac{\a\sqrt{\kappa}}{ \L} \right )
  \frac{\norm{G_{1/\L}(\w_k)}^2}{\a^2} + \frac{2\sqrt{\kappa}
  (F(\x_{k+1})-F(\x_k))}{\a} =: \delta^2,
\end{align*}
where the last inequality follows from \eqref{eq:Nesterov_1}.
Finally, by Lemma~\ref{lem:intersection_ball_ideal}, in which
we take
$\lambda = 1-\kappa^{-1/2}$ (and hence $\lambda(1-\lambda) =
\kappa^{-1/2} - \kappa^{-1}$),
we have 
\begin{align*}
\norm{\y_{k+1}-\x^*}^2 &\le \frac{1}{\sqrt{\kappa}} \left ( \left
  (1 - \frac{\a}{\L}
  \right ) \frac{\norm{G_{1/\L}(\w_k)}^2}{\a^2} - \frac{2(F(\x_{k+1})-F(\x^*))}{\a}
  \right )\\
&\hphantom{=} \quad\mbox{} + \left (1-\frac{1}{\sqrt{\kappa}} \right ) \left (
  \tilde{\sigma}_k^2 - \frac{2(F(\x_k)-F(\x^*))}{\alpha} \right )\\
& \hphantom{=}\quad\mbox{} - \left ( \frac{1}{\sqrt{\kappa}} - \frac{1}{\kappa} \right )
  \biggl ( (\kappa + \sqrt{\kappa}) \norm{\x_k-\w_k}^2 +
  \frac{2\sqrt{\kappa}  (F(\x_{k+1})-F(\x_k)) }{\a}\\
  &\hphantom{=}\quad\quad\quad\quad\mbox{}+ \left (
  1 + \frac{\a \sqrt{\kappa}}{ \L} \right)
  \frac{\norm{G_{1/\L}(\w_k)}^2}{\a^2} \biggr )\\
&  = \left ( 1-\frac{1}{\sqrt{\kappa}} \right ) \tilde{\sigma}_k^2 -
  \frac{2}{\alpha} \left ( F(\x_{k+1})-F(\x^*) \right ) - \left (
  \sqrt{\kappa} - \frac{1}{\sqrt{\kappa}} \right ) \norm{\x_k-\w_k}^2\\
& = \tilde{\sigma}_{k+1}^2 - \frac{2(F(\x_{k+1})-F(\x^*))}{\a};
\end{align*}
thereby completing the induction. 
\end{proof}

\section{Non-strongly convex and nonconvex setting}\label{sec:IA-cvx}

We turn to the composite setting
\begin{equation}\min_{\x}~ F(\x) := f(\x) + \Psi(\x), \label{eq:convex_program} \end{equation}
where $f: \R^n \to \R$ is an $L$-smooth, possibly nonconvex function and $\Psi: \R^n \to \oR$ is a closed, proper convex function. As before, we construct a general algorithmic framework for which geometric descent, Nesterov's accelerated method \cite{intro_lect}, and Lanczos method for solving trust region problem \cite{CG_Lanczos} all approximate and in the process derive a potential-like function for all three algorithms. We suppose throughout this section that the minimum of \eqref{eq:convex_program} occurs at some point denoted by $\x^*$. 
In the case of multiple minimizers, we allow $\x^*$ to be any minimizer, but we require the choice of $\x^*$ to be fixed.
We consider separately the cases that $f$ is convex (but not assumed to be strongly convex) and that $f$ is nonconvex. 

	\begin{algorithm}[ht!]
          \textbf{Initialization:} Fix $\x_0 \in \R^n$. \\
\textbf{Set:} $\z_0 = \x_0$, and any affine subspace $\mathcal{M}_1 \supseteq \z_0 + \text{span} \{ G_{1/\L}(\z_0) \}$\\
\textbf{for} $k = 1, 2, \hdots $
\begin{enumerate}
\item \textbf{Compute $\y_k$:} Let
\[ \y_k = \argmin_{\y} \{ \norm{\y-\x^*}^2 \, : \, \y \in \mathcal{M}_k\}. \]
\item \textbf{Compute $\x_k$:} Choose $\x_{k}$ such that
\begin{minipage}{0.85 \textwidth} \begin{equation} \label{eq:IA_x_convex} \x_k \in \argmin_{\x}
    \{ F(\x) \, : \, \x \in \mathcal{M}_k\}. \end{equation} \end{minipage}
\item \textbf{Compute $\z_k$:} Select an auxiliary $\z_k \in
  \mathcal{M}_k$ so that 
\begin{minipage}{0.85 \textwidth} \begin{equation} \label{eq:IA_z_convex}  G_{1/\L}(\z_k)^T(\y_k-\z_k) \ge 0 \quad \text{and} \quad
-G_{1/\L}(\z_k)^T(\x_k-\z_k)\le 0,
\end{equation} \end{minipage}
\qquad where $\bar{\z}_k = \z_k - \frac{1}{\L}G_{1/\L}(\z_k)$.
\item \textbf{Update affine subspace:}
\begin{align*} \mathcal{M}_{k+1} \supseteq \{\z_k + &\text{span}\{G_{1/\L}(\z_k)\} \} \cup \{\y_k + \text{span}\{G_{1/\L}(\z_k)\}\}\\
&\cup \{\x_k + \text{span}\{G_{1/\L}(\x_k)\}\}. 
\end{align*}
\end{enumerate}
\textbf{end}
		\caption{Idealized algorithmic framework for constraints (IA) in non-strongly convex and nonconvex settings}
		\label{alg:IA_convex}
	\end{algorithm}

Algorithm~\ref{alg:IA_convex} is similar to Algorithm~\ref{alg:IA}; the main distinction between the two is the subspace (Step 4). Algorithm~\ref{alg:IA_convex}, for the convex setting, requires a slightly larger subspace than in the strongly convex setting. Note that a solution to \eqref{eq:IA_z_convex} always exists as proved in Lemma~\ref{lem: exist_z}.

Our first analysis covers the case that $f$ is convex.
In this case, Lemma~\ref{lem:add_bound} states that \eqref{eq:IA_z_convex} implies
\begin{equation}
    \label{eq:IA_convex_Fdesc}F(\bar{\z}_k) \le F(\x_k) - \frac{1}{2\L} \norm{G_{1/\L}(\z_k)}^2.
\end{equation}

Also for the convex case, we define a potential-like quantity as follows for any $k \ge 1$:
\begin{equation}
    \Phi_k = \norm{\y_k-\x^*}^2 + \frac{k(k+1)(F(\x_k)-F(\x^*))}{2\L}.
    \label{eq:Phi_def_conv}
\end{equation}
A similar potential was used in Nesterov's original analysis \cite{nest_orig}
of accelerated gradient (for the non-composite case).

\begin{thm}[Convergence of IA, convex and nonconvex] Fix a point $\x_0 \in \R^n$. Assume $F=f+\Psi$ where $f$ is $\L$-smooth and $\Psi$ is closed, proper and convex. Then the iterates $\{\x_k, \y_k, \z_k\}_{k \ge 1}$ generated by Algorithm~\ref{alg:IA_convex} satisfy for $k = 1, 2, \hdots$
satisfy 
\begin{equation} F(\bar{\x}_k) \le F(\bar{\x}_{k-1}) - \frac{1}{2\L} \norm{G_{1/\L}(\x_k)}^2, \label{eq:IA_nonconvex} \end{equation}
where, as usual, $\bar{\x}_k = \x_k - \frac{G_{1/\L}(\x_k)}{\L}$
If, in addition, the function $f$ in \eqref{eq:convex_program} is convex, then the iterates $\{\x_k, \y_k, \z_k\}_{k \ge 1}$ generated by Algorithm~\ref{alg:IA_convex} satisfy for $k = 1, 2, \hdots$
\begin{equation}
 \Phi_{k+1} \le \Phi_k.
 \label{eq:Phi_decr}
 \end{equation}
\end{thm}

\begin{proof} We begin by proving \eqref{eq:IA_nonconvex}. 
First, we observe the point $\bar{\x}_{k-1} \in \mathcal{M}_k$ and the point $\x_k$ is the minimizer of $F$ over $\mathcal{M}_k$. Lemma~\ref{lem:smooth} with $\x = \x_k$ and these observations gives
\[ F(\bar{\x}_k) \le F(\x_k) - \tfrac{1}{2\L} \norm{G_{1/\L}(\x_k)}^2 \le F(\bar{\x}_{k-1}) - \tfrac{1}{2\L} \norm{G_{1/\L}(\x_k)}^2;\]
thus we obtain the desired result. 

For the remainder of this proof, we assume the function $f$ is convex. We define the following quantity
\[ \tilde{\y}_k = \y_k - \gamma_k G_{1/\L}(\z_k), \]
for some $\gamma_k \in \R$ that will be chosen later. 
 Since $\bar{\z}_k \in \mathcal{M}_{k+1}$, $F(\x_{k+1}) \le F(\bar{\z}_k)$.  Combining this with \eqref{eq:IA_convex_Fdesc} implies 
\begin{align} \label{eq:IA_convex_1}
    F(\x_{k+1}) - F(\x_k) \le F(\bar{\z}_k)-F(\x_{k}) \le -\frac{1}{2\L} \norm{G_{1/\L}(\z_k)}^2.
\end{align}
Moreover from Lemma~\ref{lem:smooth} with $\x = \z_k$ and $\y = \x^*$, we conclude 
\begin{align}
    \label{eq:IA_convex_2} F(\x_{k+1}) - F(\x^*) \le F(\bar{\z}_k)-F(\x^*) \le G_{1/\L}(\z_k)^T(\z_k-\x^*) - \frac{\norm{G_{1/\L}(\z_k)}^2}{2\L}.
\end{align}
We combine equations~\eqref{eq:IA_convex_1} and \eqref{eq:IA_convex_2} to conclude
\begin{equation}\begin{aligned} \label{eq:IA_convex_bound}
    & \frac{(k+1)(k+2)(F(\x_{k+1})-F(\x^*))}{2\L} 
    - \frac{k(k+1)(F(\x_k)-F(\x^*))}{2\L}\\
    & \qquad \qquad \qquad \qquad = \frac{k(k+1)(F(\x_{k+1})-F(\x_k))}{2\L} + \frac{2(k+1)(F(\x_{k+1})-F(\x^*))}{2\L}\\
    & \qquad \qquad \qquad \qquad \le \frac{2(k+1)}{2\L} G_{1/\L}(\z_k)^T(\z_k-\x^*)-\frac{(k+1)(k+2)}{2\L} \left ( \frac{\norm{G_{1/\L}(\z_k)}^2}{2\L} \right ).
\end{aligned}
\end{equation}
We use this result to compute the following:
\begin{equation}
    \label{eq:IA_potential_nonstrong}
 \begin{aligned}
& \norm{\tilde{\y}_k-\x^*}^2 = \norm{\y_k-\x^*}^2 - 2\gamma_k G_{1/\L}(\z_k)^T (\y_k-\x^*) + \gamma_k^2 \norm{G_{1/\L}(\z_k)}^2\\
& \qquad \qquad + \frac{k(k+1)(F(\x_k)-F(\x^*))}{2\L} -  \frac{(k+1)(k+2)(F(\x_{k+1})-F(\x^*))}{2\L}\\
& \qquad \qquad + \frac{(k+1)(k+2)(F(\x_{k+1})-F(\x^*))}{2\L} - \frac{k(k+1)(F(\x_k)-F(\x^*))}{2\L}\\
&\quad \le \Phi_k - 2\gamma_k G_{1/\L}(\z_k)^T (\y_k-\x^*) + \gamma_k^2 \norm{G_{1/\L}(\z_k)}^2 -  \frac{(k+1)(k+2)(F(\x_{k+1})-F(\x^*))}{2\L}\\
& \qquad \qquad + \frac{2(k+1)}{2\L} G_{1/\L}(\z_k)^T(\z_k-\x^*)-\frac{(k+1)(k+2)}{2\L} \left ( \frac{\norm{G_{1/\L}(\z_k)}^2}{2\L} \right ).
\end{aligned} 
\end{equation}
We choose $\gamma_k$ so that we can combine the two inner product terms, namely $\gamma_k = \frac{k+1}{2\L}$. This choice of $\gamma_k$ yields $\gamma_k^2 = \frac{(k+1)^2}{2\L} \cdot \frac{1}{2\L}$ and hence
\[ 2\L \cdot \gamma_k^2 - \frac{(k+1)(k+2)}{2\L} = \frac{(k+1)^2}{2\L} - \frac{(k+1)(k+2)}{2\L} \le 0. \]
This together with the first expression in \eqref{eq:IA_z_convex} yields
\begin{align*}
    2\gamma_k G_{1/\L}(\z_k)^T(\z_k-\y_k) \le 0 \quad \text{and} \quad \left (2\L \cdot \gamma_k^2 - \frac{(k+1)(k+2)}{2\L} \right ) \frac{\norm{G_{1/\L}(\z_k)}^2}{2\L} \le 0. 
\end{align*}
Substituting in \eqref{eq:IA_potential_nonstrong} yields,
\[ \norm{\tilde{\y}_k-\x^*}^2 \le \Phi_k - \frac{(k+1)(k+2)(F(\x_{k+1})-F(\x^*))}{2\L}. \]
Since $\y_{k+1}$ is the optimizer of $\norm{\y-\x^*}$ over $\y \in \mathcal{M}_{k+1}$ and $\tilde{\y}_k \in \mathcal{M}_{k+1}$, then $\y_{k+1}$ is at least as close to $\x^*$ as $\tilde{\y}_k$. Hence, we conclude 
\[ \norm{\y_{k+1}-\x^*}^2 \le \norm{\tilde{\y}_k-\x^*}^2 \le \Phi_k - \frac{(k+2)(k+1)(F(\x_{k+1})-F(\x^*))}{2\L}. \]
The result follows. 
\end{proof}

The preceding theorem yields the optimal convergence rate for
both convex and nonconvex functions as shown in the following 
corollaries.

\begin{cor}
Assume $f$ is convex and $\L$-smooth, and $F=f+\Psi$, where
$\Psi$ is closed, proper and convex.  Then for all $\x_k$ generated by Algorithm~\ref{alg:IA_convex}, the following holds for all $k \ge 1$ 
\[F(\x_k)-F(\x^*)\le \frac{2(\L\norm{\x_0-\x^*}+F(\x_0)-F(\x^*))}{k(k+1)}.\]
\label{cor:IA_cvx_rate}
\end{cor}

\begin{proof}
It follows from \eqref{eq:Phi_decr} that
$\Phi_k\le \Phi_1$.  Then substituting the definition
\eqref{eq:Phi_def_conv} and noting that
$\norm{\y_1-\x^*}\le \norm{\x_0-\x^*}$ and
$F(\x_1)-F(\x^*)\le F(\x_0)-F(\x^*)$ yields the
result.
\end{proof}

In the case of a nonconvex $f$,
recall from Lemma~\ref{lem:near_stationarity} that the
norm of the prox-gradient serves as a measure of
nearness to stationarity.  The following corollary
shows that this measure converges to 0 in
the previous algorithm.

\begin{cor}
Assume that $f$ is $\L$-smooth (but is not necessarily convex),
and $F=f+\Psi$, where $\Psi$ is closed, proper and convex.
Let $F_*$ denote $\inf_{\x\in\R^n} F(\x)$, and assume
$F_*>-\infty$. Then for all $\x_k$ generated by Algorithm~\ref{alg:IA_convex}, the following holds for all $k \ge 0$ 
\[ \min_{j=0,\ldots,k} \norm{G_{1/\L}(\x_j)}
\le\sqrt{\frac{2\L(F(\x_0)-F_*)}{k+1}}.\]
\label{cor:IA_nonconvex}
\end{cor}

The bounds in the two previous corollaries are optimal
for the class of functions under consideration for
methods that use only first-derivative information.  Indeed,
they are the best possible even for the more
restricted class of non-composite functions
(i.e., $\Psi\equiv 0$).
The optimality of the first bound is proven in
\cite{intro_lect} and the second in \cite{Cartis2010}.

\subsection{Geometric descent}\label{sec:nonconvex} We now develop a variant of geometric descent. Unlike the strongly convex setting as developed in \cite{GD_Bubeck} and \cite{GD_chen}, we do not maintain two balls in whose intersection the optimal solution $\x^*$ lives. Instead, we only have one ball whose radius guarantees enough of a decrease. Moreover, the strongly convex setting has a computable potential, which we cannot show in this case. Nonetheless, we can still show the iterates decrease at the optimal convergence rate of $\mathcal{O}(1/k^2)$.

	\begin{algorithm}[ht!]
          \textbf{Initialization:} Fix $\z_0, \y_1 \in \R^n$\\
\textbf{for} $k = 1, 2, \hdots$
\begin{enumerate}
\item Set the quantities 
\[ \bar{\z}_{k-1} = \z_{k-1} - \frac{G_{1/\L}(\z_{k-1})}{\L} \quad
  \text{and} \quad \gamma_k = \frac{k+1}{2\L}. \]
\item \textbf{Compute $\z_k$:} Find a $\z_k \in \text{aff} \{
  \bar{\z}_{k-1}, \y_k\}$ such that 
\begin{minipage}{0.85 \textwidth}\begin{equation} \label{eq: gd_z_convex} 
  G_{1/\L}(\z_k)^T(\y_k-\z_k) \ge 0
  \text{ and }
  {-G_{1/\L}}(\z_k)^T(\bar{\z}_{k-1}-\z_k)\le 0. \end{equation} \end{minipage}
\item \textbf{Compute $\y_k$:} Set 
\begin{minipage}{0.85 \textwidth} \begin{equation} \label{eq:GD_y_convex} \y_{k+1} = \y_k - \gamma_k G_{1/\L}(\z_k). \end{equation} \end{minipage}
\end{enumerate}
\textbf{end}
		\caption{Geometric descent for the non-strongly convex and nonconvex composite settings}
		\label{alg:geometric_descent_convex}
	\end{algorithm}
	\bigskip
As previously, we use the Chen et al.\ line search procedure
implicit in Lemma~\ref{lem: exist_z} to solve \eqref{eq: gd_z_convex}.  And also as previously, it follows from
Lemma~\ref{lem:add_bound} that \eqref{eq: gd_z_convex} implies \eqref{eq:IA_convex_Fdesc} with the point $\x_k$ replaced with $\bar{\z}_{k-1}$.

\begin{thm}[Convergence of GD, convex case] Assume $F=f+\Psi$ where $f$ is $\L$-smooth and convex and $\Psi$ is closed, proper and convex.  Suppose $\z_0, \y_1 \in \R^n$. Then the iterates $\{\z_k,\y_k\}$ generated by Algorithm~\ref{alg:geometric_descent_convex} satisfy for all $k \ge 1$
\begin{equation} \norm{\y_{k+1}-\x^*}^2 + \frac{(k+1)(k+2)}{2\L} (F( \bar{\z}_k)-F(\x^*)) \le \norm{\y_k-\x^*}^2 + \frac{k(k+1)}{2\L} (F(\bar{\z}_{k-1})-F(\x^*)). 
\label{eq:potdecr_gd_cvx}
\end{equation}
\end{thm}

\begin{proof} The proof is similar to that of \eqref{eq:IA_convex_bound} in the IA.  In particular, we have by \eqref{eq:GD_y_convex} for any $k \ge 1$ the following
\begin{equation} \begin{aligned} \label{eq:gd_convex_4}
    \norm{\y_{k+1}-\x^*}^2 &= \norm{\y_k-\x^*}^2 - 2 \gamma_k G_{1/\L}(\z_k)^T(\y_k-\x^*) + \gamma_k^2 \norm{G_{1/\L}(\z_k)}^2\\
    & \qquad + \frac{k(k+1)}{2\L} \left (F(\bar{\z}_{k-1})-F(\x^*) \right) - \frac{(k+1)(k+2)}{2\L} \left ( F(\bar{\z}_k) - F(\x^*) \right ) \\
    & \qquad + \frac{k(k+1)}{2\L} (F(\bar{\z}_k)-F(\bar{\z}_{k-1})) + \frac{2(k+1)}{2\L} (F(\bar{\z}_k)-F(\x^*)), 
\end{aligned}\end{equation}
where we added and subtracted the terms $\frac{(k+1)(k+2)}{2\L} (F(\bar{\z}_k)-F(\x^*))$ and $\frac{k(k+1)}{2\L} (F(\bar{\z}_{k-1})-F(\x^*))$. Inequality \eqref{eq:IA_convex_Fdesc} ensures that
\begin{equation} \label{eq:gd_convex_5a} F(\bar{\z}_k) - F(\bar{\z}_{k-1}) \le - \frac{1}{2\L} \norm{G_{1/\L}(\z_k)}^2 
\end{equation}
while Lemma~\ref{lem:smooth} assures
\begin{equation}\label{eq:gd_convex_5b}
 F(\bar{\z}_k) -F(\x^*) \le G_{1/\L}(\z_k)^T(\z_k-\x^*) - \frac{\norm{G_{1/\L}(\z_k)}^2}{2\L}. \end{equation}
The choice for the scalar $\gamma_k$ guarantees that 
\[ 2L \cdot \gamma_k^2 - \frac{(k+1)(k+2)}{(2L)^2} \le 0 \quad \text{and} \quad 2\gamma_k = \frac{2(k+1)}{2L}.\]
This together with \eqref{eq:gd_convex_4}, 
\eqref{eq:gd_convex_5a}, and \eqref{eq:gd_convex_5b} proves the result. 
\end{proof}

Using the same reasoning as in Corollary~\ref{cor:IA_cvx_rate}, we immediately obtain from
\eqref{eq:potdecr_gd_cvx} the analogous result:
\begin{cor}
Let the function $f$ be $L$-smooth and convex, the function $\Psi$ be proper, closed
and convex, and $F=f+\Psi$.
If Algorithm~\ref{alg:geometric_descent_convex} is applied
to $F$, then for all $k \ge 1$, 
\[F(\bar{\z}_k)-F(\x_0)\le \frac{2(\L\norm{\x_0-\x^*}+F(\x_0)-F(\x^*))}{k(k+1)}.\]
\end{cor}

We do not immediately obtain the analog of 
Corollary~\ref{cor:IA_nonconvex} because we
have not proved a descent bound  of
the form \eqref{eq:IA_nonconvex} for GD.
We do not know whether such a bound holds; however,
there is a slight modification to GD to ensure this
bound.  In particular, after using the
Chen et al.\ line search to obtain a candidate
$\z_k$ solving \eqref{eq: gd_z_convex}, we test the inequality
\begin{equation}
    F(\bar{\z}_k)\le F(\bar{\z}_{k-1}) - \norm{G_{1/\L}(\bar{\z}_{k-1})}^2/(2\L).
    \label{eq:suff_decr}
\end{equation}
This inequality is similar \eqref{eq:IA_convex_Fdesc} except that $G_{1/\L}$ is evaluated at a different point.  
Note that the line search already computes 
$G_{1/\L}(\bar{\z}_{k-1})$, so there is very little
additional cost to check \eqref{eq:suff_decr}.
If \eqref{eq:suff_decr} holds, then sufficient
decrease according to \eqref{eq:IA_nonconvex} is obtained,
so the iteration proceeds unchanged.  On the other hand,
if \eqref{eq:suff_decr} fails, then the algorithm discards
the vector $\z_k$ determined by line search and instead
takes $\z_k=\bar{\z}_{k-1}$; clearly \eqref{eq:suff_decr}
holds if the left-hand side is changed
to $F(\overline{(\bar{\z}_{k-1})})$
(in other words, evaluate $F$ at the result
of a forward-backward step from $\bar{\z}_{k-1}$), which
ensures sufficient decrease.  Thus, we obtain
the analog of Corollary~\ref{cor:IA_nonconvex}:
\begin{cor}
Assume that $f$ is $\L$-smooth (but is not necessarily convex),
and $F=f+\Psi$, where $\Psi$ is closed, proper and convex.
Let $F_*$ denote $\inf_{\x\in\R^n} F(\x)$, and assume
$F_*>-\infty$.
Then for all $k \ge 0$, the iterates $\{\z_j\}_{j\ge 1}$ generated
by modified GD satisfy
\[ \min_{j=0,\ldots,k} \norm{G_{1/\L}(\z_j)}
\le\sqrt{\frac{2\L(F(\x_0)-F_*)}{k+1}}.\]
\end{cor}

It should be noted that the modification to GD in the preceding
paragraph is also valid in the convex case because replacing
$\z_k$ at the end of an iteration by a different iterate that
decreases the value of $F(\bar{\z}_{k})$ can only strengthen the inequality \eqref{eq:potdecr_gd_cvx}.  Therefore, the modification
described herein can be applied to either a convex or nonconvex
$f$ without harming the convergence rate.  This is useful
in practice for nonconvex objective functions that are locally convex in a neighborhood of the
optimizer (which is the usual case).

\subsection{Accelerated Gradient}
We briefly describe Nesterov's accelerated gradient for the convex setting and its relation to the idealized framework in Algorithm~\ref{alg:IA_convex}. We illustrate this relationship on a variant of the accelerated method, FISTA \cite{beck} (see Algorithm~\ref{alg:AG_convex}). We note that a similar analysis may be performed for Nesterov's accelerated methods \cite{nest_88, intro_lect, nest_orig}.  It is known (see, \text{e.g.}, \cite{nest_orig,Gupta}) that convergence arguments for Nesterov's accelerated algorithm in the non-strongly convex composite setting exists which use a potential-like function similar to \eqref{eq:Phi_def_conv}.  
	\begin{algorithm}[ht!]
	\label{alg:AG_convex}
          \textbf{Initialization:} Choose $\x_0 \in \R^n$\\
\textbf{Set:} $\w_0 = \x_0$, $\alpha_0 = 1$\\
\textit{for $k = 1, 2, \ldots $}
\begin{enumerate}
\item \textbf{Compute $\x_k$:} 
\begin{minipage}{0.85 \textwidth} \begin{equation} \label{eq:AG_x_convex} \x_k
    := \prox_{\Psi/\L}\left ( \w_{k-1}- \frac{1}{\L} \nabla
      f(\w_{k-1}) \right ) \end{equation} \end{minipage}
\item Update the stepsizes, $\alpha_k$:
\[ \alpha_k = \frac{1 + \sqrt{1+4\alpha_{k-1}^2}}{2} \]
\item \textbf{Compute $\w_k$:} 
\begin{minipage}{0.85 \textwidth} \begin{equation} \label{eq:AG_w_convex} \w_k := \x_k + \frac{\alpha_{k-1}-1}{\alpha_{k}} (\x_k-\x_{k-1}). \end{equation} \end{minipage}
\end{enumerate}
		\caption{Accelerated gradient for the convex composite setting }
	\end{algorithm}
	
For the purpose of analysis, define the following auxiliary sequence
of points:
\begin{align}
\y_0 &= \x_0, \label{eq:AG_y_initial_convex}\\
\y_k &= \x_k + \alpha_k (\w_k-\x_k). \label{eq:AG_y_convex}
\end{align}
In essence, the sequence of points $\{\y_k\}$ from \eqref{eq:AG_y_convex}
are good
approximations to the minimizer of $\norm{\y-\x^*}^2$ in the subspace
$\mathcal{M}_k$, where as the sequence of points $\{\w_k\}$
generated by Algorithm~\ref{alg:AG_convex} approximates both the minimizer of $F(\x)$ in the subspace
$\mathcal{M}_k$ and \eqref{eq:IA_z_convex} from Algorithm~\ref{alg:IA_convex}. As in GD (Algorithm~\ref{alg:geometric_descent_convex}), the points $\y_k$ can be computed by \eqref{eq:GD_y_convex}; the equivalence of \eqref{eq:GD_y_convex} and \eqref{eq:AG_y_convex} is shown for example in the works \cite{Teboulle_Drori, Kim_Fessler}. Under these identifications, it follows (see \textit{e.g.} \cite{nest_orig,Gupta}) that the potential function \eqref{eq:Phi_def_conv} satisfies for all $k \ge 1$, 
\[\Phi_{k+1} \le \Phi_k\]
provided the function $f$ is $L$-smooth and convex. The discussions in Section~\ref{sec:nonconvex}, then, show that the accelerated method achieves convergence guarantees of $\mathcal{O}(1/k^2)$. 

Whether accelerated methods are superior to gradient descent remains an open question in the nonconvex setting; their performance escaping saddle points faster than gradient descent has been observed \cite{AG_near_crit_pts,AGD_saddle_points_Jordan}. As far as we know, there are no convergence results for accelerated methods when the function $f$ is nonconvex. By either adding a boundedness assumption on the domain \cite{ghadimi_lan} or interlacing Algorithm~\ref{alg:AG_convex} with proximal gradient descent steps \cite{GL2}, Nesterov's accelerated method achieves convergence guarantees matching GD (Algorithm~\ref{alg:geometric_descent_convex}) in the nonconvex setting. We note the additional cost of computing the proximal gradient descent step in the modified Nesterov's accelerated method means both GD and AG require two proximal operations at each iteration.

\section{Trust-region Lanczos method}\label{sec:CG-Lanczos}
We consider the problem
\begin{equation} \min_{\x \in \R^n} \tfrac{1}{2} \x^TA\x - \b^T\x \quad \text{subject
    to} \,
  \norm{\x} \le \Delta, \label{eq:TR_problem} \end{equation}
where $\Delta$ is a positive constant. This problem is usually called the ``trust-region subproblem,'' as it is the main subproblem occurring in the trust-region method
for nonlinear optimization.  Because of widespread use of the trust-region method, algorithms for solving the trust-region subproblem has been intensively studied in the
literature \cite{TR_Conn,Agarwal_TR,HoNguyen_TR,Hazan_TR}. 

 Classical theory of the TRS
(see, e.g., \cite{NW} for the results and their history) establishes the following exact characterization
of the optimizer for any symmetric $A$ (positive semidefinite or not).

\begin{thm}
A point $\hat\x$ is the minimizer of \eqref{eq:TR_problem} if and only if there exists a $\mu^*$ satisfying the
following conditions:

(a) $(A+\mu^* I)\hat\x =  \b$,

(b) $A+\mu^*I$ is positive semidefinite, and

(c) Either (i) $\norm{\hat\x}\le\Delta$ and $\mu^*=0$ or (ii) $\norm{\hat\x}=\Delta$ and $\mu^*\ge 0$. 
\label{thm:TRS_kkt}
\end{thm}

A specialized method developed by Gould et al. \cite{CG_Lanczos} uses
classical conjugate gradient with Lanczos. Rather than using the
conjugate gradient basis $\{\p_0, \p_1, \ldots, \p_{k-1}\}$ for the Krylov
space, they use different basis $\{\q_0, \q_1, \ldots, \q_{k-1}\}$ generated by the Lanczos
method. As long as the conjugate-gradient iteration does not break
down, the Lanczos vectors may be recovered from the conjugate-gradient
iterates as
\[ \q_k = \frac{ \sigma_k \r_k}{\norm{\r_k}}, \quad \text{where
    $\sigma_k = -\text{sign}(\alpha_{k-1}) \sigma_{k-1}$ and
    $\sigma_0 = 1$} \]
while the Lanczos tridiagonal matrix may be expressed as 
\begin{equation} \label{eq:T_matrix}
T_k = \begin{pmatrix}
\frac{1}{\alpha_0} & \frac{\sqrt{\beta_0}}{|\alpha_0|} & & & & &\\
\frac{\sqrt{\beta_0}}{|\alpha_0|} & \frac{1}{\alpha_1} +
\frac{\beta_0}{\alpha_0} & \frac{\sqrt{\beta_1}}{|\alpha_1|}  & & &
&\\
& \frac{\sqrt{\beta_1}}{|\alpha_1|} &\frac{1}{\alpha_2} +
\frac{\beta_1}{\alpha_1} & \cdot & &
&\\
& & \cdot & \cdot & \cdot & &\\
& & & \cdot &  \frac{1}{\alpha_{k-1}} +
\frac{\beta_{k-2}}{\alpha_{k-2}} &
\frac{\sqrt{\beta_{k-1}}}{|\alpha_{k-1}|}\\
& & & & \frac{\sqrt{\beta_{k-1}}}{|\alpha_{k-1}|} & \frac{1}{\alpha_k}
  + \frac{\beta_{k-1}}{\alpha_{k-1}}
\end{pmatrix}.
\end{equation}

The idea with the Lanczos approach is to consider vectors 
\[ \x \in \text{span}\{\q_0, \q_1, \ldots, \q_k\} \]
and seek $\x_k = Q_k \h_k$ where $\x_k$ solves the problem
\[ \min_{\x \in \text{span}\{\q_0, \q_1, \ldots, \q_k\}}~ \frac{1}{2}
  \x^TA\x - \b^T\x \quad \text{subject to} \quad \norm{\x} \le \Delta. \]
This is equivalent to finding $\h_k$ which solves
\begin{equation} \label{eq:TR_subproblem}
\min_{\h \in \R^{k+1}} \tfrac{1}{2}\h^TT_k \h - \norm{\r_0} \h^T\e_1 \text{
  subject to } \norm{\h} \le \Delta,
\end{equation}
where $\e_1 = (1, 0, \ldots, 0)^T$. 
For the remainder of this section,
let $f(\x)=\x^TA\x/2-\b^T\x$ be the objective function, let
$B=\{\x:\norm{\x} \le \Delta\}$ be the feasible region,
and let $F(\x)=f(\x)+\mathrm{i}_B(\x)$ be the composite
objective function. Here, $\mathrm{i}_B(\x)$ is the indicator
function of $B$.

The trust-region Lanczos algorithm is presented twice below, first as
Algorithm~\ref{alg:CG} and again in a slightly more abstracted
form as Algorithm~\ref{alg:CG_simplified}.  
One can view Algorithm~\ref{alg:CG} as minimizing the function $F$
over an increasing subspace whose basis consists of the vectors
$\{\p_0, \p_1, \ldots, \p_{k-1}\}$. For convenience, define $\mathcal{V}_k =
\text{span}\{\p_0, \ldots, \p_{k-1}\}$ for $k = 0, 1, \ldots$. Then the
trust region with Lanczos can be understood more abstractly as
in Algorithm~\ref{alg:CG_simplified}.

We require that these iterations do not break down, i.e., $\r_k$ is nonzero and independent of $\{\r_0,\ldots,\r_{k-1}\}$ on every iteration. 
This result is already stated in the by \cite{CG_Lanczos} in the case that $A$ is positive definite.
Here we extend this result to the semidefinite case under the assumption of exact arithmetic (which is made throughout this paper, but is not made by \cite{CG_Lanczos}).

\begin{lem}
Assume $A$ is positive semidefinite.
Then the numerator and denominator of \eqref{eq:alpha} and \eqref{eq:update_residual_conjugate} are positive on each iteration until the algorithm terminates at the optimizer, except possibly the denominator of \eqref{eq:alpha} may vanish on the final iteration.
\end{lem}

\noindent
{\bf Remark.}  If the denominator of \eqref{eq:alpha} vanishes on the final iteration, then the algorithm's progress is not impeded because the final (optimal) $\x_{k+1}$ is obtained by solving \eqref{eq:TR_subproblem}, which does not involve $\alpha_k$.

\begin{proof}
First, consider the classical conjugate gradient algorithm for solving a symmetric positive definite linear system $C\x=\d$ initialized at $\bz$.  This algorithm terminates after precisely $m$ iterations \cite{GolubVanLoan}, where $m$ is defined as follows.  Let the eigenfactorization of $C$ be $C=QDQ^T$, where $Q$ is orthogonal and $D$ is diagonal.  Let $\bar\d=Q^T\d$.  Then $m$ is the number of  distinct diagonal entries of $D$, counting only those diagonal entries of $D$ that correspond to nonzero entries of $\bar\d$.  On all iterations prior to termination, the numerator and denominator of \eqref{eq:alpha} and \eqref{eq:update_residual_conjugate}  are positive.

Recall also that classical CG implicitly computes a 
degree $m-1$ polynomial $p$ such that on iteration $m$ (the terminal iteration), 
$\x^*=p(C)\d$, so $Cp(C)\d=\d$.  On the other hand
$Cp(C)\d=QDQ^Tp(QDQ^T)\d=QDp(D)\bar\d$.  Therefore,
$Cp(C)\d=\d$ is rewritten $Dp(D)\bar\d=\bar\d$.  The matrix on the left-hand side is diagonal, and therefore the 
equation $\lambda_ip(\lambda_i)\bar d_i=\bar d_i$ for each $i$ holds.
This implies that for all $i$ such that $\bar d_i$ is nonzero, $p(\lambda_i)=1/\lambda_i$.

Next, consider the trust-region Lanczos method applied to \eqref{eq:TR_problem} for positive semidefinite $A$.  
Let $m$ in this case be the number of distinct positive eigenvalues of $A$ such that $\b$ has a nonzero component in the corresponding eigenvector direction when $\b$ is expressed in the eigenvector basis of $A$.

We observe that the updates to $\r_k$ and $\p_k$ are identical to those of classical CG when projected into $\mathcal{R}(A)$, the range of $A$.  If we write $\b=\b^\parallel+\b^\perp$, the decomposition of $\b\in\R^n$ into $\mathcal{R}(A)\oplus\mathcal{N}(A)$, then it is apparent from the recursions that $\r_k$ for each $k$ will have component in $\mathcal{N}(A)$ equal to $\b^\perp$ on each iteration, while the component of $\p_k$ in $\mathcal{N}(A)$ is of the form $\sigma\b^\perp$ for a $\sigma\ge 1$ (since the $\beta_k$'s are all positive).
Now we consider two cases, namely, $\b^\perp=\bz$ or $\b^\perp\ne \bz$.  If $\b^\perp=\bz$, then the algorithm will terminate with an exact solution on iteration $m$.  This is because the none of the vectors $\r_k$ or $\p_k$ ever depart from $\mathcal{R}(A)$, so the algorithm is in correspondence with the positive definite case applied to a lower dimensional space for which the result of \cite{CG_Lanczos} applies.  In this case, the numerators and denominators of all the fractions are positive.

If $\b^\perp\ne \bz$, then the algorithm requires one extra step, i.e., it terminates at step $m+1$.  Let $\mu^*$ be the optimal KKT multiplier.
Note that in this case, it is necessary that $\mu^*>0$ (because there is no solution $\x^*$ to $A\x^*=\b$ under the assumption of this case).  To prove that the algorithm terminates on this step, observe that it can select any $\x_{m+1}$ that may be written as $p(A)\b$, where $p$ is now a polynomial of degree $m$ (instead of $m-1$).  In particular, there is such a polynomial such that $p(\lambda_i)=1/(\lambda_i+\mu^*)$ for the $m$ distinct eigenvalues of $A$ that correspond to nonzero entries of $\b$ (in the eigenvector basis), plus the additional condition that $p(0)=1/\mu^*$.  These conditions together mean that 
$(A+\mu^*I)p(A)\b=\b$, in other words, there is a solution in the Krylov subspace to (a) of Theorem~\ref{thm:TRS_kkt}.  This solution also necessarily solves (b) and (c).
\end{proof}


It should be noted
that \cite{CG_Lanczos} actually considers a more general
problem in which $A$ is an arbitrary (not necessarily positive semidefinite).  Furthermore, their treatment allows preconditioning.  In addition,
the Krylov iteration is restarted in the Gould et al.\ algorithm for the ``hard case'';
a comment on this matter appears at the end of
the next section.   In this
section, only the case when $f$ is convex (i.e., $A$ is positive
semidefinite) is under consideration, and the ``hard case'' does
not occur when $A$ is positive semidefinite.

\begin{algorithm}[ht!]
          \textbf{Initialization:} Choose $\x_0 = \bz$.
\textbf{Set:} $\r_0 := \b$ and $\p_0 := \r_0$ \\
\textbf{for} $k = 0, 1, 2, \ldots $\\
\quad Set
\begin{equation} \label{eq:alpha} \alpha_k = \frac{\r_k^T\r_k}{\p_k^TA\p_k} \end{equation}
\quad Construct matrix $T_k$ from $T_{k-1}$ using \eqref{eq:T_matrix}\\
\quad \textbf{If} $\norm{\x_k + \alpha_k \p_k} \ge \Delta$, then 
\begin{center} solve the 
tridiagonal trust region subproblem \eqref{eq:TR_subproblem} to obtain
$\h_k$\\ and set
$\x_{k+1} := Q_k \h_k$ \end{center}
\quad \textbf{else} 
\[\x_{k+1} := \x_k + \alpha_k \p_k\]
\textbf{end}\\
Update the residual and conjugate direction
\begin{equation} \begin{aligned} \label{eq:update_residual_conjugate}
\r_{k+1} &:= \r_k-\alpha_k A \p_k \\
\beta_k &:= \frac{\r_{k+1}^T\r_{k+1}}{\r_k^T\r_k}\\
\p_{k+1} &:= \beta_k \p_k + \r_{k+1}
\end{aligned}
\end{equation}
		\caption{Trust-region Lanczos}
		\label{alg:CG}
	\end{algorithm}
	\bigskip

\bigskip
\begin{algorithm}[ht!]
          \textbf{Initialization:} Choose $\x_0 := \bz$\\
\textbf{Set:} $\r_0 := \b$, $\p_0 := \r_0$, and $\mathcal{V}_1 = \text{span}\{\p_0\}$ \\
\textbf{for} $k = 1, 2, \ldots $\\
\quad Solve the following subproblem
\begin{equation} \label{eq: subspace_subprob} \x_k \in \argmin_{\x \in \mathcal{V}_k} \left\{ \tfrac{1}{2} \x^TA\x - \b^T\x: \quad \norm{\x} \le \Delta\right\}.
\end{equation}
\quad Update the residual $\r_{k+1}$ and conjugate direction $\p_{k+1}$
using \eqref{eq:update_residual_conjugate}.\\
\quad Set $\mathcal{V}_{k+1} = \text{span} \{\p_0, \p_1, \ldots, \p_k\}$. \\
\textbf{end}
		\caption{Trust-region Lanczos (abstract version)}
		\label{alg:CG_simplified}
	\end{algorithm}

\begin{lem}\label{lem:CG_recurrence} The following hold for $k \ge 0$:
\begin{align*}
\r_{k+1} = \b - \sum_{j=0}^k \alpha_j A \p_j \quad \mbox{ and } \quad 
A\p_k \in \text{\rm span} \{\p_0, \p_1, \ldots, \p_k, \p_{k+1}\}.
\end{align*}
\end{lem}

\begin{proof} We will show this by induction. For $k = 0$, the
  definition of $\r_k$ in \eqref{eq:update_residual_conjugate} gives
  $\r_1 = \b - \alpha_0 A \p_0$ and $\p_1 = \beta_0 \p_0 + \b - \alpha_0
  A\p_0$. Rearranging one obtains $\alpha_0 A\p_0 = \beta_0 \p_0 + \b -
  \p_1$. Since $\b = \p_0$, it follows that $A\p_0 \in \text{span} \{ \p_0,
  \p_1\}$, thereby proving the result for $k =0$. Next, suppose $\r_{k+1} =
  \b - \sum_{j=0}^{k} \alpha_j A \p_j$ and $A\p_k \in \text{span} \{ \p_0,
  \p_1, \ldots, \p_k, \p_{k+1}\}$ for $k \ge 0$. By definition of
  $\r_{k+2}$, we obtain
\[\r_{k+2} = \r_{k+1} - \alpha_{k+1} A \p_{k+1} = \b - \sum_{j=0}^{k+1}
  \alpha_j A \p_j\]
  and the result for the residual immediately follows. By the
  definition of $\p_{k+2}$, we have
\[\p_{k+2} = \beta_{k+1} \p_{k+1} + \r_{k+2} = \beta_{k+1} \p_{k+1} + \b -
  \sum_{j=0}^k \alpha_j A \p_j - \alpha_{k+1} A \p_{k+1}\]
\vspace{-0.75cm}
\[\Rightarrow \quad \alpha_{k+1} A \p_{k+1} = \beta_{k+1} \p_{k+1} + \b - \p_{k+2} - \sum_{j=0}^k\alpha_j A \p_j. \]
Because $A\p_j \in \text{span} \{\p_0, \p_1, \ldots, \p_k, \p_{k+1}\}$ for
$0 \le j \le k$ and $\b = \p_0$, we obtain that $A\p_{k+1} \in
\text{span} \{ \p_0, \p_1, \ldots, \p_k, \p_{k+1}, \p_{k+2}\}$, thereby
completing the induction. 
Note that the preceding argument relied on the fact that $\alpha_{k+1}>0$ (and hence we can divide by it), which is a consequence of the no-breakdown property of the algorithm discussed above.
\end{proof}

We present definitions of $\y_k$, $\z_k$ for trust-region Lanczos  method
which show that it implements the idealized algorithm. Existence of these
sequences automatically implies that the method satisfies the
convergence bound developed in Sections~\ref{sec:IA1} and \ref{sec:IA-cvx}. This is
surprising since the method does not have prior knowledge of $\x^*$
nor does it explicitly compute the forward-backward step at any
iteration. Yet, we show that the forward-backward step, in this case the projection 
of a steepest descent step over
the ball, always lies in the subspace spanned by the Krylov
space. Let $\{\x_k\}_{k \ge 0}$ and $\{\mathcal{V}_k\}_{k
  \ge 1}$ be the sequence of iterates and sequence of subspaces
respectively generated by Algorithm~\ref{alg:CG_simplified}. For the purpose of
analysis of Algorithm~\ref{alg:CG_simplified}, define the following auxiliary
sequences of vectors
\begin{equation} \begin{aligned} \label{eq:CG_iterates}
\y_k &= \argmin_{\y} \{ \norm{\y-\x^*}^2 \, : \, \y \in
  \mathcal{V}_k \},\quad  y_0 = z_0, \quad \mbox{ and }\\ 
\z_k &\in \mathcal{V}_k \quad \text{is any vector satisfying
  $G_{1/\L}(\z_k)^T(\y_k-\z_k) \ge 0$}\\
 & \qquad \qquad \text{and $F(\bar{\z}_k) \le F(\x_k) -
  \frac{1}{2\L} \norm{G_{1/\L}(\z_k)}^2$ and $\z_0 = \x_0$.}
\end{aligned} \end{equation}
Note that $\z_k$ may be constructed using \eqref{eq:z_rep}.

\begin{lem} \label{lem:subspace_property_holds} Suppose Algorithm~\ref{alg:CG_simplified} generates a sequence of
  iterates $\{\x_k\}_{k \ge 0}$ and a corresponding increasing
  sequence of subspaces $\{\mathcal{V}_k\}_{k
  \ge 1}$. Define the sequences
$\{\y_k\}_{k \ge 0}$ and $\{\z_k\}_{k \ge 0}$ as in
  \eqref{eq:CG_iterates}. Then for all $k \ge 1$
  \begin{equation} \label{eq:subspace_property} \z_{k-1} + \text{\rm span}\{\y_{k-1} - \z_{k-1},
    G_{1/\L}(\z_{k-1})\} \subset \mathcal{V}_k. 
    \end{equation} \end{lem}

\begin{proof}  We prove the existence of the sequence $\{\z_k\}_{k\ge 0}$
  and the subspace properties \eqref{eq:subspace_property} by
  induction. For notational convenience, we define the affine space
\[ \bar{\mathcal{M}}_k = \z_{k-1} + \text{span} \{ \y_{k-1} - \z_{k-1}, G_{1/\L}(\z_{k-1})\}.\]
A simple computation shows $\bar{\mathcal{M}}_1 = \text{span}
\{G_{1/\L}(\bz)\}$ under the assumption that $\y_0
  = \z_0 = \bz$. We observe $G_{1/\L}(\bz) \in \text{span} \{-\b\} =
  \text{span} \{\p_0\} = \mathcal{V}_1$. Therefore it follows that
  $\bar{\mathcal{M}}_1 \subset \mathcal{V}_1$. Lemma~\ref{lem: exist_z}, then,
  implies the existence of a $\z_1$ satisfying the property in
  \eqref{eq:CG_iterates}. 

Next, suppose $\z_k$ exists and $\bar{\mathcal{M}}_k \subset
\mathcal{V}_k$ for all $k \ge 1$, so we can define $\bar{\mathcal{M}}_{k+1}$. To complete the induction, we need to show $\z_k + \text{span}\{\z_k-\y_k, G_{1/\L}(\z_k)\} \subset \mathcal{V}_{k+1}$. Since
$\mathcal{V}_k$ are increasing subspaces, we know $\z_k, \y_k \in
\bar{\mathcal{M}}_k \subset \mathcal{V}_k \subset \mathcal{V}_{k+1}$. 
Hence we need to prove $G_{1/L}(\z_k) \in \mathcal{V}_{k+1}$. By definition,
\[G_{1/\L}(\z_k) = L(\z_k - \text{proj}_B(\z_k-\nabla f(\z_k)/\L)).\]
In order to complete the induction, we need to show that $\nabla f(\z_k) \in \mathcal{V}_{k+1}$. First, we have $\nabla f(\z_k) = A\z_k -\b$. Because $\z_k \in \mathcal{V}_k$, there exist coefficients $\gamma_i \in \R$ for $i = 0, \ldots k-1$ such that 
\begin{align*}
\nabla f(\z_k) = A\z_k -\b = \sum_{i=0}^{k-1} \gamma_i A \p_i - \b.
\end{align*}
Lemma~\ref{lem:CG_recurrence} says for any $i =0, 1, 2, \ldots, k-1$ we have $A \p_i \in \mathcal{V}_{i+1} \subset V_{k+1}$. Moreover, we have $\b \in \mathcal{V}_{k+1}$.  Hence, it follows that $\nabla f(\z_k) \in \mathcal{V}_{k+1}$. Therefore, we have $\bar{\mathcal{M}}_{k+1} \subset
\mathcal{V}_{k+1}$ and by Lemma~\ref{lem: exist_z} the iterate $\z_{k+1}$ exists. 
\end{proof}

With this lemma in hand, we can construct the potential function for the trust-region Lanczos method.  In this theorem, $\L$ denotes the largest eigenvalue of $A$ and $\alpha\ge 0$ the smallest.

\begin{thm} Suppose Algorithm~\ref{alg:CG_simplified} generates a sequence of
  iterates $\{\x_k\}_{k \ge 0}$ and a corresponding increasing
  sequence of subspaces $\{\mathcal{V}_k\}_{k
  \ge 1}$. Define the sequences
$\{\y_k\}_{k \ge 0}$ and $\{\z_k\}_{k \ge 0}$ as in
  \eqref{eq:CG_iterates}. Then the iterates satisfy for all $k \ge 1$ the following linear convergence rate
\begin{align*}
\norm{\y_{k+1}-\x^*}^2 + \frac{ 2(F(\x_{k+1})-F(\x^*))}{\a} \le
  \left ( 1- \sqrt{ \frac{\a}{\L}} \right ) \left (
\norm{\y_k-\x^*}^2 + \frac{ 2(F(\x_{k})-F(\x^*))}{\a} 
  \right ),
\end{align*}
and the following sublinear convergence rate
\[
F(\x_k)-F(\x^*)\le \frac{2(L\norm{\x^*} + F(\bz)-F(\x^*))}{k(k+1)}.
\]
\label{thm:CG_convergence}
\end{thm}

Note that the linear rate depends on $\alpha$ and loses its strength in the case that $\alpha=0$ (i.e., that $A$ is positive semidefinite but not positive definite).  See further remarks on this matter in the next section.  Also, note that the sublinear rate is based on Corollary~\ref{cor:IA_cvx_rate} under the further assumption that $\x_0=\bz$.

\begin{proof} The previous Lemma~\ref{lem:subspace_property_holds}
  shows $\z_{k-1} + \text{\rm span}\{\y_{k-1} - \z_{k-1},
    G_{1/\L}(\z_{k-1})\} \subset \mathcal{V}_{k+1}$ for all $k \ge 1$. Hence, the
    iterates of the idealized algorithm are the iterates of the CG
    under the identification $\mathcal{M}_k = \mathcal{V}_k$ for all
    $k \ge 1$ where $\mathcal{M}_k$ are defined in Algorithm~\ref{alg:IA}. By Theorem~\ref{thm:converge_IA} and Corollary~\ref{cor:IA_cvx_rate}, the result follows. 
\end{proof}

\section{Chebyshev polynomial analysis of the trust-region Lanczos method}
\label{sec:cheb}

In this section, we present an alternative analysis of trust-region Lanczos method
based on Chebyshev polynomial approximation.  Chebyshev polynomial approximation
is the basis of Daniel's \cite{Daniel} classical
proof of linear convergence
of conjugate gradient.  We present two Chebyshev analyses, one for the strongly
convex case, i.e., a positive definite $A$, and one for the convex case,
i.e., a positive semidefinite $A$.  The first analysis yields a linear
convergence rate of the form $(1-\sqrt{\a/\L})^k$ and the second
a sublinear rate of the form $1/k^2$.  Note that although the
sublinear rate is clearly inferior asymptotically to the linear rate,
it is nonetheless possible in the strongly
convex case when $\a\ll \L$ that the sublinear rate
outperforms the linear rate for small or medium values of $k$.

This analysis hinges on the eigenvalues 
$\a=\lambda_1\le\cdots\le\L=\lambda_n$ of $A$.  Because
the original problem is invariant under orthogonal change
of coordinates, we assume that a change of coordinates has
been applied so that $A$ is diagonal.  To emphasize this
assumption, we write $D$ instead of $A$ for the coefficient
matrix in this section: $f(\x)=\x^TD\x/2-\b^T\x$.
Furthermore, without loss of generality, the diagonal entries
of $D$ are in increasing order $\lambda_1\le\cdots \le \lambda_n$.
Although this section focuses on the convex case, i.e.,
$\lambda_1\ge 0 $, some of the
results derived here apply also to the nonconvex case as noted
at the end of this section.

Fix an iteration $k$.  
The trust-region Lanczos method finds $\x_k$ such that
$\x_k$ minimizes $f(\x)$ over the constraint set $\Vert\x\Vert\le \Delta$ and
$\x_k\in\Span\{\b,D^2\b,\ldots,D^{k-1}\b\}$, the dimension-$k$ Krylov
space.  The second constraint means
that $\x_k$ can be written as $p_k(D)\b$, where $p_k$ is a polynomial of
degree at most $k-1$.  In other words, $(\x_k)_i=p_k(\lambda_i)b_i$ for $i=1,\ldots,n$.

Let $\mu^*$ denote the KKT multiplier appearing
in Theorem~\ref{thm:TRS_kkt}.  Assume that $\mu^*>0$, so that the solution is on the
boundary.  In the case when $\mu^*=0$, the Gould et al.\
method reduces to classical (unconstrained) conjugate
gradient which has already been analyzed by \cite{Daniel}.
Thus, the following KKT conditions are satisfied by the optimizer $\x^*$:
\begin{align*}
  (D+\mu^*I)\x^*-\b &= \bz, \\
  \Vert\x^*\Vert &=\Delta.
\end{align*}
The first equation is written $(\lambda_i+\mu^*)x^*_i=b_i$ for $i=1,\ldots,n$, i.e.,
$$x_i^*=\frac{b_i}{\lambda_i+\mu^*}.$$

We would like the error to be bounded above  in terms of
$f(\x_0)-f(\x^*)$.
We start with an equality for
the optimal value of the problem:
\begin{align}
  f(\x_0)-f(\x^*) &= f(\bz)-f(\x^*) \notag \\
  & = -f(\x^*) \notag \\
  & = -\frac{1}{2}(\x^*)^TD\x^* + \b^T\x^* \notag \\
  & = \frac{1}{2}(\x^*)^T(D+\mu^*I)\x^* - (\x^*)^T((D+\mu^*I)\x^*-\b) \notag\\
  & = \frac{1}{2}(\x^*)^T(D+\mu^*I)\x^*.\label{eq:fxsbd1}
\end{align}
Since $(\x^*)^TD\x^*\ge 0$ and $(\x^*)^T\x^*=\Delta^2$, it follows from \eqref{eq:fxsbd1}
that
\begin{equation}
  f(\x_0)-f(\x^*)\ge \mu^*\Delta^2/2.
  \label{eq:fxsbd2}
\end{equation}

\subsection{Construction of two families of Chebyshev polynomials}
We now construct two families of Chebyshev polynomials useful
for the rest of the analysis.
Let $T_k(z)$ be the degree-$k$ first-kind Chebyshev polynomial.
Among the useful properties of the first-kind Chebyshev polynomial
are: $T_k([-1,1])\subset [-1,1]$ and $T_k(\pm 1)\in \{-1,1\}$.

For the first family, define
$$\hat T_k(z)=\left\{\begin{array}{ll}
1+T_k(z), & \mbox{$k$ even}, \\
1-T_k(z), & \mbox{$k$ odd}.
\end{array}
\right.
$$
Thus, $\hat T_k([-1,1])\subset [0,2]$, and furthermore, $\hat T_k(-1)=2$.
In addition, $\hat T_k'(-1)<0$, hence by monotonicity of Chebyshev polynomials
outside their natural domain, $\hat T_k(z)>2$ for all $z<-1$.
Next, define
$$q^{\rm A}(t)=c_k\hat T_k\left(\frac{2t-(\lambda_1+\lambda_n+2\mu^*)}{\lambda_n-\lambda_1}
\right)$$
where $c_k$ is chosen so that $q^{\rm A}(0)=1$.  We can get an estimate for $c_k$ as follows.  Let
\begin{equation}
    \zeta=\sqrt{\frac{\lambda_n+\mu^*}{\lambda_1+\mu^*}}.
    \label{eq:zetadef}
\end{equation}
For the case of even $k$, 
\begin{align*}
    1 &=q^{\rm A}(0)\\
    &= c_k\hat T_k\left(-\frac{\zeta^2+1}{\zeta^2-1}\right)\\ &=c_k\left(1+T_k\left(-\frac{\zeta^2+1}{\zeta^2-1}\right)\right) \\
    &=c_k\left(1 +\frac{1}{2}\left(\left(\frac{\zeta+1}{\zeta-1}\right)^k+\left(\frac{\zeta-1}{\zeta+1}\right)^k\right)\right) \\
    &\ge \frac{c_k}{2}\left(\frac{\zeta+1}{\zeta-1}\right)^k.
\end{align*}
The fourth line used a standard closed-form expression for Chebyshev polynomials (see, e.g., \cite{TrefethenBau}), while the
last line dropped two of the three terms in the fourth line.
The preceding inequality is rearranged into
\begin{equation}
    c_k\le 2\left(\frac{\zeta-1}{\zeta+1}\right)^k.
    \label{eq:c_kbd}
\end{equation}
The same estimate holds for odd $k$ with a similar analysis.

In sum, $q^{\rm A}(t)$ is a degree-$k$ polynomial
whose constant coefficient is 1. 
Furthermore, $q^{\rm A}(\lambda_1+\mu^*)=c_k \hat T(-1)=2c_k$
and $q^{\rm A}(\lambda_n+\mu^*)=c_k \hat T(1)\in\{0,2c_k\}$.
Also,
$q^{\rm A}(t)\ge 0$ over the interval $[\lambda_1+\mu^*,\lambda_n+\mu^*]$,
and $q^{\rm A}(t)\le 2c_k$ over this interval,
where $c_k$ is bounded as in \eqref{eq:c_kbd}.

Finally, one observes $q^{\rm A}(t)<1$ for $t\in[\lambda_1+\mu^*,\lambda_n+\mu^*]$;
this follows because 
$q^{\rm A}(\lambda_1+\mu^*) = c_k\hat  T(-1)=2c_k$, and
$1=q^{\rm A}(0)>q^{\rm A}(\lambda_1+\mu^*)$ by
the monotonicity as noted above, so $1>2c_k$, i.e., $c_k<1/2$.

For the second family of Chebyshev polynomials,
let
$$\dot T_k(z)=\left\{\begin{array}{ll}
1+T_k(z), & \mbox{$k$ odd},\\
1-T_k(z), & \mbox{$k$ even}.
\end{array}
\right.
$$
Then $\dot T_k([-1,1])\subset [0,2]$ and $\dot T_k(-1)=0$.  Furthermore,
$\dot T_k'(-1)=k^2$, and  $\dot T_k'(z)\le k^2$ for $z\in[-1,1]$.  The
latter two properties follow because $T_k'$ is a scaled Chebyshev polynomial
of the second kind whose properties are well known.
Now define
$$q^{\rm B}(t) =\frac{(\lambda_n+\mu^*)\dot T_{k+1}(2t/(\lambda_n+\mu^*)-1)}
{2(k+1)^2t}.$$
Since the numerator is a polynomial of $t$ of degree $k+1$ that vanishes at the origin,
the above quotient is also a polynomial, and we can determine its
value at the origin using L'H\^opital's rule, obtaining $q^{\rm B}(0)=1$.
By nonnegativity of $\dot T_{k+1}$ over $[-1,1]$, we observe that
$q^{\rm B}$ is nonnegative on the interval $[0,\lambda_n+\mu^*]$.
Moreover, the derivative of the numerator with respect to $t$
over this interval is bounded
above by $2(k+1)^2$, whereas the derivative of the denominator is equal to $2(k+1)^2$,
which means that $q^{\rm}(t)\le 1$ over $[0,\lambda_n+\mu^*]$.

Thus,  we see that $q^{\rm A}$, $q^{\rm B}$ have the following common properties:
both are polynomials of degree $k$ whose constant coefficient is 1, and
both take on values in $[0,1]$ over the interval $[\lambda_1+\mu^*,\lambda_n+\mu^*]$.
The distinction is that $q^{\rm A}$ shrinks over this interval exponentially
fast with $k$, but the base of the exponent tends to 1
when $\zeta \gg 1$, i.e., $\lambda_1+\mu^*\ll \lambda_n+\mu^*$.  On the other hand, $q^{\rm B}$
shrinks more slowly over this interval, but it shrinks 
independently of $\lambda_1$.  (Note that the construction of $q^{\rm B}$
did not involve $\lambda_1$.)

The next part of the analysis applies to both
$q^{\rm A}$ and $q^{\rm B}$, so let $q$ denote either of these.
Since $q(t)$ is a degree-$k$ polynomial whose constant coefficient
is 1, there exists a degree-$(k-1)$ polynomial denoted $\tilde q(t)$
such that $q(t)=t\tilde q(t)+1$.
Let $\y=-\tilde q(D+\mu^*I)\b$.   Observe first that $\y$ is feasible:
\begin{align*}
  \Vert \y\Vert^2 &= \sum_{i=1}^n y_i^2  \\
  &= \sum_{i=1}^n\tilde q(\lambda_i+\mu^*)^2b_i^2  \\
  & = \sum_{i=1}^n\frac{(q(\lambda_i+\mu^*)-1)^2}{(\lambda_i+\mu^*)^2}\cdot b_i^2  \\
  & = \sum_{i=1}^n (q(\lambda_i+\mu)-1)^2 \cdot \left(\frac {b_i}{\lambda_i+\mu^*}\right)^2 \\
  &\le \sum_{i=1}^n \left(\frac {b_i}{\lambda_i+\mu^*}\right)^2 \\
  & = \sum_{i=1}^n (x_i^*)^2 \\
  &= \Delta^2.
\end{align*}
The fifth line follows because $q(t)\in[0,1]$ for all
$t\in[\lambda_1+\mu^*,\lambda_n+\mu^*]$.

Next, we compare the objective value of $\x_k$ with that of $\x^*$, noting that
by the optimality property of the trust-region Lanczos algorithm, $f(\x_k)\le f(\y)$:
\begin{align*}
  f(\x_k)-f(\x^*) &\le   f(\y)-f(\x^*) \\
   &= \y^TD\y/2-\b^T\y - ((\x^*)^TD\x^*/2-\b^T\x^*) \\
  & = \y^T(D+\mu^*I)\y/2 - \b^T\y - ((\x^*)^T(D+\mu^*I)\x^*/2-\b^T\x^*) \\
  & \hphantom{=}\quad\mbox{} + ((\x^*)^T\x^*-\y^T\y)(\mu^*/2) \\
  &\equiv t_1+t_2,
\end{align*}
where we now analyze the two terms separately.

\begin{align}
  t_1 &= \y^T(D+\mu^*I)\y/2 - \b^T\y - ((\x^*)^T(D+\mu^*I)\x^*/2-\b^T\x^*) \notag \\
  &=(\y-\x^*)^T((D+\mu^*I)(\x^*+\y)/2- \b) \notag \\
  &=\sum_{i=1}^n \left(-\tilde q(\lambda_i+\mu^*)b_i-\frac{b_i}{\lambda_i+\mu^*}\right)
  \left((\lambda_i+\mu^*)\left(-\tilde q(\lambda_i+\mu^*)b_i+\frac{b_i}{\lambda_i+\mu^*}\right)/2 - b_i\right)\notag \\
  &= \frac{1}{2}\sum_{i=1}^n\frac{q(\lambda_i+\mu^*)^2b_i^2}{\lambda_i+\mu^*} \label{eq:t1analysis}
\end{align}
To obtain the fourth line, we substituted $\tilde q(t)=(q(t)-1)/t$ and then simplified.

For the other term,
\begin{align}
  t_2 &=  ((\x^*)^T\x^*-\y^T\y)(\mu^*/2) \notag \\
  & = \frac{\mu^*}{2}\sum_{i=1}^n \left[\frac{b_i^2}{(\lambda_i+\mu^*)^2}-b_i^2\tilde q(\lambda_i+
  \mu^*)^2\right] \notag \\
  &=\frac{\mu^*}{2}\sum_{i=1}^n \left[\frac{b_i^2}{(\lambda_i+\mu^*)^2}-
    \frac{b_i^2(q(\lambda_i+\mu^*)-1)^2}{(\lambda_i+\mu^*)^2}\right] \notag \\
  &=\frac{\mu^*}{2}\sum_{i=1}^n \frac{b_i^2}{(\lambda_i+\mu^*)^2} \cdot (2q(\lambda_i+\mu^*)-q(\lambda_i+\mu^*)^2) \label{eq:t2analysis}
\end{align}

\subsection{Analysis of first Chebyshev construction (linear rate)}
For this subsection, we assume $q\equiv q^{\rm A}$ in \eqref{eq:t1analysis} and
\eqref{eq:t2analysis}.  
\begin{align*}
  t_1 &= \frac{1}{2}\sum_{i=1}^nq^{\rm A}(\lambda_i+\mu^*)^2(\lambda_i+\mu^*)\cdot 
  \frac{b_i^2}{(\lambda_i+\mu^*)^2}  \\
  & \le \frac{1}{2} (\max_i q^{\rm A}(\lambda_i+\mu^*)^2)
  \cdot \sum_{i=1}^n\frac{b_i^2}{(\lambda_i+\mu^*)^2}(\lambda_i+\mu^*) \\
  & \le  2c_k^2
  \sum_{i=1}^n\frac{b_i^2}{(\lambda_i+\mu^*)^2}(\lambda_i+\mu^*) \\
  & = 2c_k^2
  \sum_{i=1}^n(x_i^*)^2(\lambda_i+\mu^*) \\
  & = 2c_k^2(\x^*)^T(D+\mu^*I)\x^* \\
  & = 4c_k^2(f(\x_0)-f(\x^*)) && \mbox{(by \eqref{eq:fxsbd1}).}
\end{align*}

Also
\begin{align*}
  t_2 &=  \frac{\mu^*}{2}\sum_{i=1}^n \frac{b_i^2}{(\lambda_i+\mu^*)^2}
  \cdot (2q^{\rm A}(\lambda_i+\mu^*)-q^{\rm A}(\lambda_i+\mu^*)^2) \\
  &=  \frac{\mu^*}{2}\sum_{i=1}^n \frac{b_i^2}{(\lambda_i+\mu^*)^2}
  \cdot q^{\rm A}(\lambda_i+\mu^*)(2-q^{\rm A}(\lambda_i+\mu^*)) \\
  &\le \frac{\mu^*}{2}\cdot(\max_i q^{\rm A}(\lambda_i+\mu^*)) \cdot
  \max_i(2-q^{\rm A}(\lambda_i+\mu^*))\cdot
  \sum_{i=1}^n \frac{b_i^2}{(\lambda_i+\mu^*)^2} \\
  &  = \frac{\mu^*\Delta^2}{2}\cdot(\max_i q^{\rm A}(\lambda_i+\mu^*)) \cdot
  \max_i(2-q^{\rm A}(\lambda_i+\mu^*))\\
  &\le \frac{\mu^*\Delta^2}{2}\cdot (2c_k) \cdot 2 \\
  &= 2\mu^*\Delta^2c_k \\
  &\le 4(f(\x_0)-f(\x^*))c_k && \mbox{(by \eqref{eq:fxsbd2}).}
\end{align*}

Combining, we obtain,
$$f(\x_k)-f(\x^*)\le 4(c_k^2+c_k)(f(\x_0)-f(\x^*)).$$
Since $c_k\le 1/2$, we can use $c_k^2\le c_k/2$ to bound the right-hand side and obtain the following theorem.

\begin{thm}
The trust-region Lanczos method applied to a problem
with a positive semidefinite $A$ converges
at the following rate:
$$f(\x_k)-f(\x^*)\le 6c_k(f(\x_0)-f(\x^*)),$$
where $c_k$ has an upper bound as in
\eqref{eq:c_kbd} with $\zeta$ defined by 
\eqref{eq:zetadef}.
\end{thm}

It follows from the definition of $\zeta$ that
$\zeta\le \sqrt{\lambda_n/\lambda_1}=\sqrt{\L/\a}$
and therefore $(\zeta-1)/(\zeta+1)\le 1-\sqrt{\a/\L}$, hence
$c_k\le 2(1-\sqrt{\a/\L})^k$ by \eqref{eq:c_kbd}.
Thus, the preceding bound yields a result comparable
to Theorem~\ref{thm:CG_convergence} in the
previous section.  The bound here may be much tighter
if $\mu^*\gg \lambda_1$.  In particular, a linear rate of convergence is still obtained when $\a=0$ provided $\mu^*>0$.

\subsection{Analysis of second Chebyshev construction (sublinear rate)}

For this subsection, we assume $q(t)\equiv q^{\rm B}(t)$ in \eqref{eq:t1analysis} and
\eqref{eq:t2analysis}. Then
\begin{align*}
  t_1 &=\frac{1}{2}\sum_{i=1}^n\frac{q^{\rm B}(\lambda_i+\mu^*)^{2}b_i^2}{\lambda_i+\mu^*} \\
  &= \frac{\lambda_n+\mu^*}{2}\sum_{i=1}^n\frac{q^{\rm  B}(\lambda_i+\mu^*)(\lambda_i+\mu^*)}
  {\lambda_n+\mu^*}\cdot q^{\rm B}(\lambda_i+\mu^*)\frac{b_i^2}{(\lambda_i+\mu^*)^2} \\
  &=\frac{\lambda_n+\mu^*}{2}\sum_{i=1}^n\frac{\dot T_{k+1}(2(\lambda_i+\mu^*)/(\lambda_n+\mu^*)-1)}
  {2(k+1)^2}\cdot q^{\rm B}(\lambda_i+\mu^*)\frac{b_i^2}{(\lambda_i+\mu^*)^2} \\
  &\le \frac{\lambda_n+\mu^*}{2}\sum_{i=1}^n\frac{2}{2(k+1)^2}\cdot 1\cdot \frac{b_i^2}{(\lambda_i+\mu^*)^2} \\
  & = \frac{(\lambda_n+\mu^*)\Delta^2}{2(k+1)^2}
\end{align*}

Next,
\begin{align*}
  t_2 &=\frac{\mu^*}{2}\sum_{i=1}^n \frac{b_i^2}{(\lambda_i+\mu^*)^2} \cdot
  q^{\mathrm{B}}(\lambda_i+\mu^*)(2-q^{\rm B}(\lambda_i+\mu^*)) \\
  &\le \mu^*\sum_{i=1}^n \frac{b_i^2}{(\lambda_i+\mu^*)^2} \cdot q^{\mathrm{B}}(\lambda_i+\mu^*) \\
  &= \mu^*\sum_{i=1}^n\frac{b_i^2}{(\lambda_i+\mu^*)^2} \cdot
  \frac{(\lambda_n+\mu^*)\dot T_{k+1}(2(\lambda_i+\mu^*)/(\lambda_n+\mu^*)-1)}
       {2(k+1)^2(\lambda_i+\mu^*)} \\
  &= \frac{\lambda_n+\mu^*}{2(k+1)^2}\sum_{i=1}^n \frac{b_i^2}{(\lambda_i+\mu^*)^2} \cdot
       \dot T_{k+1}(2(\lambda_i+\mu^*)/(\lambda_n+\mu^*)-1)\cdot\frac{\mu^*}{(\lambda_i+\mu^*)} \\
  &\le \frac{\lambda_n+\mu^*}{2(k+1)^2}\sum_{i=1}^n \frac{b_i^2}{(\lambda_i+\mu^*)^2}
       \cdot 2 \cdot 1  \\
  &= \frac{(\lambda_n+\mu^*)\Delta^2}{(k+1)^2}.
\end{align*}
Combining the two results derived above, 
we obtain 
\begin{align*}
f(\x_k)-f(\x^*)&\le \frac{3(\lambda_n+\mu^*)\Delta^2}{2(k+1)^2} \\
&\le \frac{3(\L\Delta^2 + 2(f(\x_0)-f(\x^*)))}{2(k+1)^2} \\
&\le \frac{3(\L\Vert\x_0-\x^*\Vert^2 + 2(f(\x_0)-f(\x^*)))}{2(k+1)^2}.
\end{align*}
We have recovered essentially the same bound as in 
Theorem~\ref{thm:CG_convergence}.

We conclude this section with a few remarks on the
nonconvex case of the trust-region subproblem, that
is, the case that $\lambda_1<0$.  In this case,
much of the analysis of this section goes through unchanged.
The analysis shows that the method still converges at a linear rate
of convergence to the optimizer except in the hard
case discussed below.  Furthermore, the
optimizer found is a global minimizer rather than
an approximate stationary point in the case of 
general $F$.  

The difficulty in establishing a
rate is that there is no prior positive lower
bound on $\lambda_1+\mu^*$ in the nonconvex case;
it can be an arbitrarily small positive number, so
the linear rate can be arbitrarily poor.  It is
even possible that $\lambda_1+\mu^*=0$.
The case that $\lambda_1+\mu^*=0$ is called the
``hard case.''  The hard case can also be encountered
on intermediate steps of the trust-region Lanczos method
when $\lambda_1<0$.
When the hard case occurs on intermediate steps, 
Algorithm~\ref{alg:CG_simplified} is no longer an accurate
description of the Gould et al.\ algorithm because in that
case, their method leaves the current Krylov space by adding
a perturbation in a different direction.  Thus, the convergence of
the trust-region Lanczos method in the general nonconvex case
appears to be unknown.

A different algorithm for the TRS proposed by Hazan and Koren \cite{HazanKoren}, and later simplified and improved by Wang and Xia \cite{WangXia}, attains the optimal sublinear rate of $1/k^2$ in all cases including the hard case.  These two algorithms both require a preliminary approximate computation of $\lambda_1$, the smallest eigenvalue of $A$, and therefore use information beyond the gradient/prox model considered herein.  Yet another first-order approach also able to handle the hard case was very recently proposed by Beck and Vaisbourd \cite{BeckVaisbourd}. Another recent approach based on classical matrix iteration is due to Adachi et al.\ \cite{Adachi} (also handling the hard case).

\section{Conclusions}
We have showed that several algorithms for constrained and composite optimization can be analyzed by placing them in a framework of an idealized algorithm and then arguing that they decrease a potential at a certain rate.  One algorithm in this class is the trust-region Lanczos algorithm; our analysis yields rate that can also be found via classical Chebyshev polynomial approximation theory.

The trust-region Lanczos algorithm represents an extension of classical conjugate gradient to a constrained problem. An immediate open question is what rate is possible for that method in the nonconvex case, especially the hard case. In particular, the methods mentioned in the previous section with $O(1/k^2)$ nonconvex rates do not have the strong optimality property of the trust-region Lanczos method (i.e., that each iteration finds the exact minimizer in a growing sequence of subspaces) in the convex case.

A broader open question is whether conjugate gradient extends to constrained or composite optimization of minimizing a quadratic function in which the constraint term $\Psi$ is not ellipsoidal.

\bibliographystyle{plain}
\bibliography{bibliography}
\end{document}